\documentclass[reqno,12pt]{amsart}
\usepackage{amsmath,amsfonts,amsthm,amssymb}
\usepackage{tikz}
\usepackage[utf8]{inputenc}
\usepackage[T1]{fontenc}
\usepackage{palatino}
\usetikzlibrary{cd,arrows,decorations.pathmorphing,decorations.markings,backgrounds,positioning,fit,calc,shapes,patterns,external}
\usepackage{setspace}
\usepackage{fullpage}
\usepackage{mathtools}
\usepackage[shortlabels]{enumitem}
\usepackage[subrefformat=parens, labelfont=up]{subcaption}
\usepackage{hyperref}
\usepackage{comment}

\newtheorem{thm}{Theorem}[section]

\newtheorem{innercustomthm}{Theorem}
\newenvironment{mythm}[1]
  {\renewcommand\theinnercustomthm{#1}\innercustomthm}
  {\endinnercustomthm}

\newtheorem{lemma}[thm]{Lemma}

\newtheorem{cor}[thm]{Corollary}
\newtheorem{corollary}[thm]{Corollary}
\newtheorem{prop}[thm]{Proposition}
\theoremstyle{definition}
\newtheorem{defn}[thm]{Definition}

\theoremstyle{remark}

\theoremstyle{remark}

\usepackage{manfnt}
\newenvironment{ex}{\refstepcounter{thm}\begin{proof}[Example \emph{\thethm}]}{\end{proof}}
\newenvironment{rem}{\refstepcounter{thm}\begin{proof}[Remark \emph{\thethm}]}{\end{proof}}

\newenvironment{warn}{\refstepcounter{thm}\begin{proof}[Warning \emph{\thethm}]}{\end{proof}}

\numberwithin{equation}{section}

\newcommand{\seqnum}[1]{\href{https://oeis.org/#1}{\rm \underline{#1}}}

\colorlet{dark purple}{red!35!blue}
\colorlet{dark green}{green!70!black}
\colorlet{dark red}{red!80!black}
\colorlet{dark blue}{blue!80!black!80!cyan}

\tikzstyle{mutable}=[inner sep=0.9mm,circle,draw,minimum size=2mm,fill=white]
\tikzstyle{frozen}=[inner sep=.9mm,rectangle,draw,fill=white]
\tikzstyle{dot} = [fill=black!25,inner sep=0.5mm,circle,draw,minimum size=.5em]
\tikzstyle{marked}=[inner sep=0.5mm,circle,draw,blue!75!black,fill=blue!50]

\tikzset{bpoint/.style = {shape=circle,fill=black,draw,minimum size=.5em, inner sep=0}}
\tikzset{wpoint/.style = {shape=circle,thick,fill=white,draw,minimum size=.5em, inner sep=0}}

\def\Spec{\operatorname{Spec}}
\def\CA{\mathcal{A}}

\def\SS{\Sigma}

\allowdisplaybreaks[1]

\title{Point Counts of Cluster Varieties of Marked Surfaces over Finite Fields}  

\author{James Beyer}
\address{Instituto de Matem{\'a}ticas, Universidad Nacional Aut{\'o}noma de M{\'e}xico \\ Ciudad Universitaria, CDMX, M{\'e}xico}
\email{jebeyer@im.unam.mx}
\date{2026-10-20}

\keywords{cluster algebra, cluster variety, finite field, point count, deep point}
\subjclass[2020]{
Primary 13F60, 
Secondary 05E14, 
14G15
}

\begin{document}

\begin{abstract}
	We establish formulae for point counts of cluster varieties of cluster algebras of marked surfaces, possibly with punctures. We then establish formulae for the number of non-deep points in these cluster varieties over $\mathbb{F}_2$, which gives us the number of algebraic tori necessary to cover the cluster manifold over $\mathbb{F}_2$. We also show that these formulae satisfy certain recurrence relations. 
\end{abstract}

\maketitle 

\tableofcontents

\section{Introduction}

Point counts of cluster varieties over finite fields were first examined by Chapoton \cite{Chap11, Chap15}. Various aspects of cluster varieties over finite fields have since been discussed by Benito, Muller, Rajchgot, and Smith \cite{BMRS15}, Lam and Speyer \cite{LS22}, Galashin and Lam \cite{GL24skein, GL24catalan}, and P{\'e}rez Melesio and Simental \cite{PMS26}. In this note, we discuss cluster varieties of cluster algebras corresponding to marked surfaces with boundary. 

We begin by finding the number of points in a cluster variety associated to a polygon with boundary coefficients. This formula is a natural extension to formulae of Chapoton \cite{Chap11} and Lam and Speyer \cite{LS23}, but can be computed using only the number of vertices and the size of the field.

We then show that for an annulus $\Sigma_{0,2,m}$ with $m$ marked points and at least one marked point on each boundary component, the corresponding cluster variety over $\mathbb{F}_q$ \textemdash the field of $q=P^n$ elements, for some prime $P$ and positive integer $n$ \textemdash is comprised of \[ \#\mathcal{A}(\Sigma_{0,2,m}; \mathbb{F}_q) = (q-1)^{m} \big( q^{m} + (-1)^{m} \big) \] points. This formula provides a generalization of the Jacobsthal-Lucas numbers, with the $q=2$ case coinciding with the usual Jacobsthal-Lucas numbers. 

Building upon these results, we derive the following formula for the number of points in the cluster variety of a marked surface over $\mathbb{F}_q$. Note that for any fixed surface, this formula is a polynomial. 
\begin{mythm}{\ref{thm: punctured count full}}
    Let $\Sigma_{g,b,m}^p$ be a connected, triangulable marked surface of genus $g \geq 0$ with $b\geq 1$ boundary components, $m\geq 2$ marked points, and $p\geq 0$ punctures. 
	The corresponding cluster variety over $\mathbb{F}_q$ contains 
	\[ \#\mathcal{A}(\Sigma_{g,b,m}^p; \mathbb{F}_q) = (q+1)^{2g+b-2} (q-1)^{2g+m+b-2} \big( q^{2g+m+b+p-2} (q^2+q-1)^p + (-1)^m (2q^2-1)^p \big) \]
	points.
\end{mythm}
We also show in Lemma \ref{lemma: recurrence} that, when $q=2$, this formula satisfies the recurrence relation of the Jacobsthal sequence, 
\[ \#\mathcal{A}(\Sigma_{g,b,m}^p; \mathbb{F}_2) = \#\mathcal{A}(\Sigma_{g,b,m-1}^p; \mathbb{F}_2) + 2 \#\mathcal{A}(\Sigma_{g,b,m-2}^p; \mathbb{F}_2). \]

We then compute the number of non-deep points in $\mathcal{A}(\Sigma_{g,b,m}^p ; \mathbb{F}_2)$. Since an algebraic torus over $\mathbb{F}_2$ is set-theoretically a point, the number of non-deep points is the number of algebraic tori necessary to cover the cluster manifold over $\mathbb{F}_2$.
Finally, we show that these numbers satisfy the following recurrence relation. 
\begin{mythm}{\ref{thm: nondeep recurrence}}
    Let $\Sigma_{g,b,m}^p$, $\Sigma_{g,b,m-1}^p$, and $\Sigma_{g,b,m-2}^p$ be connected, triangulable marked surfaces of genus $g$ with $b\geq 1$ boundary components, $m-2 \geq 2$ marked points, and $p\geq 1$ punctures, such that none are punctured bigons. 
    Then the number of non-deep points over $\mathbb{F}_2$ satisfies the recurrence relation 
    \[ L(\Sigma_{g,b,m}^p) = L(\Sigma_{g,b,m-1}^p) + 2 L(\Sigma_{g,b,m-2}^p) + 2^{2g+b+2p-1}. \]
\end{mythm}

\section{Cluster algebras of marked surfaces}

Cluster algebras were first introduced by Fomin and Zelevinsky \cite{FZ02}. For a gentle introduction to the theory of cluster algebras, we recommend the reader begin with Fomin, Williams, and Zelevinsky's book \cite{FWZChapters123, FWZChapters45}. 
Cluster algebras of surfaces were introduced by Gekhtman, Shapiro, and Vainshtein \cite{GSV05}, and were further developed by Fomin, Shapiro, and Thurston \cite{FST08, FT18}. They have since been studied by many authors, including Musiker, Schiffler, and Williams \cite{MSW11,MW13,MSW13}, Thurston \cite{Thu14}, Muller \cite{Mul16}, Mandel and Qin \cite{MQarxiv}, and Gei\ss, Labardini-Fragoso, and Wilson \cite{GLFWarxiv}, among others. 

Let $\Sigma_{g,b,m}^p$ denote a smooth, oriented, compact-with-boundary surface of genus $g$ with $b$ boundary components and $m$ marked points on the boundary, with at least one marked point on each boundary component, and $p$ marked points in the interior which we will refer to as \emph{punctures}. We may omit the superscript and simply write $\Sigma_{g,b,m}$ when $p=0$. We will refer to marked disks as \emph{polygons} and denote them as $\Delta_m \coloneq \Sigma_{0,1,m}$.

A \emph{marked arc} in $\Sigma_{g,b,m}^p$ is an immersion of an interval such that the endpoints are either at marked points or punctures, while the interior of the interval is mapped into the interior of the surface; a marked arc is \emph{simple} if it has no crossings and is not contractible. A collection of marked arcs is called \emph{compatible} if the arcs are simple and do not cross in the interior of $\Sigma_{g,b,m}^p$. 
A \emph{triangulation} is a collection of compatible marked arcs which, along with the boundary arcs, divides the surface into a union of topological triangles; we say that a surface $\Sigma_{g,b,m}^p$ is \emph{triangulable} if it admits a triangulation. A triangulation consists of \[ n(\Sigma_{g,b,m}^p) = 6g + 3b + 3p + m - 6 \] non-boundary arcs \cite[Proposition 2.10]{FST08}, so a surface is triangulable if and only if $n(\Sigma_{g,b,m}^p)$ is non-negative. If $p\geq 1$, we use tagged triangulations as defined by Fomin, Shapiro, and Thurston \cite{FST08}. 

We assign a variable to each homotopy class of simple marked (tagged) arcs of $\Sigma_{g,b,m}^p$, which we refer to as a \emph{cluster variable}. The set of cluster variables associated to a triangulation is called a \emph{cluster}; together, the triangulation and its associated cluster are called a \emph{seed}. We assume that boundary arcs correspond to \emph{frozen} variables, which are contained in every cluster. All other cluster variables will correspond to \emph{mutable} variables. Any two triangulations are related by a sequence of flips; likewise, any two clusters are related by a sequence of \emph{mutations} which replace a cluster variable with the cluster variable corresponding to the flipped arc. Relations among cluster variables are generated by skein relations among the marked arcs. 

Let $\mathcal{X}$ be the set of all cluster variables associated to simple marked (tagged) arcs in $\Sigma_{g,b,m}^p$. Fix a seed with cluster $ \mathbf{x} = \{x_1, \dotsc, x_n\}$, and let $\mathcal{F} = \mathbb{Q}(x_1, \dotsc, x_n)$ be the field freely generated by that cluster. The \emph{cluster algebra} $A(\Sigma_{g,b,m}^p)$ is the $\mathbb{Z}$-subalgebra of $\mathcal{F}$ generated by $\mathcal{X}$ along with the inverses of all frozen variables. An important property of cluster algebras is the Laurent phenomenon \textemdash given any cluster $ \mathbf{x} = \{x_1, \dotsc, x_n\}$, all cluster variables in $\mathcal{X}$ can be written as Laurent polynomials in the variables of that cluster, so we have an embedding \[ A(\Sigma_{g,b,m}^p) \hookrightarrow \mathbb{Z}[x_1^{\pm 1}, \dotsc, x_n^{\pm 1}]. \] 

Given a field $\Bbbk$, a cluster algebra has an associated \emph{cluster variety} over $\Bbbk$, in general defined as  $\mathcal{A}(\Sigma_{g,b,m}^p ; \Bbbk) \coloneq \Spec_\Bbbk A(\Sigma_{g,b,m}^p)$. 
We can identify the $\Bbbk$-points of $\mathcal{A}(\Sigma_{g,b,m}^p ; \Bbbk)$ with ring homomorphisms $ A(\Sigma_{g,b,m}^p) \to \Bbbk$. These homomorphisms are distinguished by the values assigned to the cluster variables, or in our setting, to the homotopy classes of simple marked arcs. We will discuss the structure of the cluster variety further in Section \ref{section: tori}.

\section{Covering pairs}\label{section: locally acyclic}

Let $\Sigma_{g,b,m}^p$ be a smooth, oriented, triangulable marked surface of genus $g\geq 0$ with $b\geq 1$ boundary components, $p\geq 0$ punctures, and $m\geq 2$ marked points on the boundary, with at least one marked point on each boundary component. The corresponding cluster algebra $A(\Sigma_{g,b,m})$ is \emph{locally acyclic} \cite[Theorem 10.6]{Mul13}, ensuring that it has many desirable properties. In particular, $A(\Sigma_{g,b,m}^p)$ is finitely generated, integrally closed, and locally a complete intersection \cite[Theorem 4.2]{Mul13}. A precise definition of local acyclicity is outside the scope of this note, but we will recall some relevant results here.

First, recall that an \emph{ice quiver} is a tuple $Q = (V,F,E,s,t)$, where $V$ is the set of \emph{vertices}, $F \subseteq V$ is the set of \emph{frozen} vertices, $V\smallsetminus F$ is the set of \emph{mutable} vertices, $E$ is the set of \emph{arrows}, $s\colon E \to V$ assigns a \emph{source} vertex to each arrow, and $t\colon E \to V$ assigns a \emph{target} vertex to each arrow. Further, we require that $s(\alpha) \neq t(\alpha)$ for all $\alpha \in E$, and if $s(\alpha_1) = t(\alpha_2)$ for some $\alpha_1, \alpha_2 \in E$, then $s(\alpha_2) \neq t(\alpha_1)$. More succinctly, $Q$ is a directed graph with no loops or two-cycles, where a subset of vertices has been designated as frozen. 
The mutable part $Q^{mut}$ of an ice quiver $Q$ is the subquiver constructed from all mutable vertices along with the arrows that both start and end at mutable vertices. 

It is common for cluster algebras to be defined by the mutation equivalence class of an ice quiver. We will briefly discuss construction of an ice quiver from a triangulation of a polygon in Section \ref{section: anticlique}. 

\begin{defn}\cite[Definition 5.2]{Mul13}
	A \emph{covering pair} $(\alpha,\beta)$ in an ice quiver $Q$ is a pair of mutable vertices $\alpha$, $\beta$ such that there is an arrow from $\alpha$ to $\beta$, but this arrow is not contained in any bi-infinite path in $Q^{mut}$. 
	
	A \emph{covering pair} $(\alpha, \beta)$ in a cluster algebra $A$ is a pair of cluster variables such that there is some seed $(Q, \mathbf{x})$ with $\alpha,\beta \in \mathbf{x}$ and $(\alpha,\beta)$ is a covering pair in $Q$.
\end{defn}

Some examples of covering pairs can be seen in Figures \ref{fig: covering pair} and \ref{fig: covering pair punctured}. 

\begin{lemma}\cite[Lemma 10.1]{Mul13}\label{lemma: Muller LA 10.1}
	Let $\alpha$ and $\beta$ be tagged arcs in $\Sigma_{g,b,m}^p$ which cut out an unpunctured triangle (as in Figure \ref{fig: covering pair}), with all endpoints in the boundary and the endpoints of $\beta$ distinct. Then $(\alpha,\beta)$ are a covering pair in $\mathcal{A}(\Sigma_{g,b,m}^p; \Bbbk)$.
\end{lemma}

\begin{figure}[htb]
    \centering
    \begin{tikzpicture}
       \path[fill=black!10] (-2,1) to (2,1) to (2,-1) to (-2,-1) to (-2,1);
        \draw[thick] (-2,1) to (2,1) (-2,-1) to (2,-1);
        \draw[dashed] (-2,-1) to (-2,1) (2,-1) to (2,1);
        \draw[-Triangle,thin] (-2,0) to (-2,0.1);
        \draw[-Triangle,thin] (2,0) to (2,0.1);
        
        \node[dot] (1) at (-1,1) {};
        \node[dot] (3) at (1,-1) {};
        \node[dot] (4) at (-1,-1) {};
        \draw[thick] (1) to (3);
        \node at (-1.22,0) {$\alpha$};
        \draw[thick] (1) to (4);
        \node at (0.4,0) {$\beta$};
    \end{tikzpicture}
    \caption{A covering pair in a marked annulus.}
    \label{fig: covering pair}
\end{figure}

\begin{lemma}\cite[Lemma 10.2]{Mul13}
    Let $\alpha$ be a tagged arc in $\Sigma_{g,b,m}^p$ which cuts out a once-punctured digon containing radius $\beta$ (as in Figure \ref{fig: covering pair punctured}). Then $(\alpha,\beta)$ are a covering pair in $\mathcal{A}(\Sigma_{g,b,m}^p; \Bbbk)$.
\end{lemma}

\begin{figure}[htb]
    \centering
    \begin{tikzpicture}
        \path[fill=black!10] (-3,1.2) to (3,1.2) to (2,-1) to (-2,-1) to (-3,1.2);
        \draw[thick] (-3,1.2) to (-2,-1) (-2,-1) to (2,-1) (2,-1) to (3,1.2);
        \draw[dashed] (-3,1.2) to (3,1.2);
        
        \node[dot] (1) at (-2,-1) {};
        \node[dot] (2) at (2,-1) {};
        \node[dot] (p) at (0,0) {};
        \draw[thick] (2) to (p);
        \draw[thick] (1) to [out=75,in=180] (0,0.8) to [out=0,in=105] (2);
        \node at (-1.22,0) {$\alpha$};
        \node at (0.4,-0.54) {$\beta$};
    \end{tikzpicture}
    \caption{A covering pair in a punctured surface.}
    \label{fig: covering pair punctured}
\end{figure}

\begin{prop}\cite[Corollary 5.4]{Mul13}
    Let $(\alpha,\beta)$ be a covering pair in an ice quiver $Q$. Then $\mathcal{A}(A(Q)[\alpha^{-1}]; \Bbbk)$ and $\mathcal{A}(A(Q)[\beta^{-1}]; \Bbbk)$ cover $\mathcal{A}(A(Q); \Bbbk)$.
\end{prop}

That is, we can cover the cluster variety $\mathcal{A}(\Sigma_{g,b,m}^p; \Bbbk)$ with the two varieties corresponding to the localizations of the cluster variables of the covering pair. Topologically, we can realize these varieties via cutting.

\begin{prop}\cite[Proposition 3.18]{BeyerDeepIIarxiv}\label{prop: cutting}
    Let $\Sigma$ be a triangulable marked surface, and let $S$ be a compatible collection of non-boundary marked arcs in $\Sigma$. Then the map $\Sigma \smallsetminus S \to \Sigma$ induces an isomorphism 
    \[ A(\Sigma \smallsetminus S)/\langle s^\prime - s^{\prime \prime}, \forall s \in S \rangle \overset{\sim}{\longrightarrow} A(\Sigma)[S^{-1} ] \]
    Here, $s^\prime$ and $s^{\prime \prime}$ denote the two preimages of $s$ in $\Sigma \smallsetminus S$.
\end{prop}

\begin{ex}\label{ex: annulus cutting}
Using the covering pair in Figure \ref{fig: covering pair}, we can cover the cluster variety associated to an annulus $\Sigma_{0,2,m}$ using two polygons with $m+2$ sides, as in Figure \ref{fig: cuttings}.
\end{ex}

\begin{figure}[htb]
    \centering
    \begin{tikzpicture}
       \path[fill=black!5] (-2,2) to (-1,2) to (0,0) to (-1,-2) to (-2,-2) to (-4,-2) to (-5,0) to (-4,2) to (-2,2);
        \draw[thick] (-2,2) to (-1,2) (-1,2) to (0,0) (0,0) to (-1,-2) (-1,-2) to (-2,-2) (-3,-2) to (-4,-2) (-4,-2) to (-5,0) (-5,0) to (-4,2) (-4,2) to (-3,2);
        \draw[dashed] (-2,-2) to (-3,-2) (-2,2) to (-3,2);
        
        \node[dot] (1) at (-1,2) {};
        \node[dot] (2) at (0,0) {};
        \node[dot] (3) at (-1,-2) {};
        \node[dot] (4) at (-4,2) {};
        \node[dot] (5) at (-5,0) {};
        \node[dot] (6) at (-4,-2) {};
        \draw[thick] (1) to (3);
        \node at (-1.22,0) {$\alpha$};
        \node at (-0.3,1.1) {$\beta^\prime$};
        \node at (-4.7,-1.1) {$\beta^{\prime \prime}$};
        \node at (-0.3,-1.1) {$\gamma$};
    \end{tikzpicture}
    \hspace{0.5cm}
    \begin{tikzpicture}
       \path[fill=black!5] (2,2) to (1,2) to (0,0) to (1,-2) to (2,-2) to (4,-2) to (5,0) to (4,2) to (2,2);
        \draw[thick] (2,2) to (1,2) (1,2) to (0,0) (0,0) to (1,-2) (1,-2) to (2,-2) (3,-2) to (4,-2) (4,-2) to (5,0) (5,0) to (4,2) (4,2) to (3,2);
        \draw[dashed] (2,-2) to (3,-2) (3,2) to (2,2);
        
        \node[dot] (1) at (1,2) {};
        \node[dot] (2) at (0,0) {};
        \node[dot] (3) at (1,-2) {};
        \node[dot] (4) at (4,2) {};
        \node[dot] (5) at (5,0) {};
        \node[dot] (6) at (4,-2) {};
        \draw[thick] (1) to (3);
        \node at (1.22,0) {$\beta$};
        \node at (4.7,-1.1) {$\alpha^\prime$};
        \node at (0.3,1.1) {$\alpha^{\prime \prime}$};
        \node at (0.3,-1.1) {$\gamma$};
    \end{tikzpicture}
    \caption{Cutting along $\alpha$ produces a polygon where two copies of $\alpha$ are nonconsecutive boundary arcs (right), and cutting along $\beta$ produces a polygon where two copies of $\beta$ are nonconsecutive boundary arcs (left).}
    \label{fig: cuttings}
\end{figure}

\section{Polygons and anticlique stratification}\label{section: anticlique}

We begin our examination of cluster varieties in the setting of polygons. Some of the results we reference are stated in terms of quivers so it will be helpful to be able to translate between triangulations and quivers. 
To construct an ice quiver $Q_T$ from a triangulation $T$ of a polygon $\Delta_m$, we place a frozen vertex on each of the $m$ boundary arcs and a mutable vertex on each of the $m-3$ non-boundary arcs. Frozen vertices are depicted with a rectangle, while mutable vertices are depicted with a circle, as in Figure \ref{fig: polygon quiver}. Within each triangle, draw arrows between the vertices in a clockwise manner, omitting arrows between frozen vertices. 

\begin{figure}[ht]
	\centering
    \begin{tikzpicture}[baseline={(0,0)}]
		\path[fill=black!5] (-2,2) to (-1,2) to (0,0) to (-1,-2) to (-2,-2) to (-4,-2) to (-5,0) to (-4,2) to (-2,2);
		\draw[thick] (-2,2) to (-1,2) (-1,2) to (0,0) (0,0) to (-1,-2) (-1,-2) to (-2,-2) (-3,-2) to (-4,-2) (-4,-2) to (-5,0) (-5,0) to (-4,2) (-4,2) to (-3,2);
		\draw[thick] (-2,-2) to (-3,-2) (-2,2) to (-3,2);
		
		\node[dot] (1) at (-1,2) {};
		\node[dot] (2) at (0,0) {};
		\node[dot] (3) at (-1,-2) {};
		\node[dot] (4) at (-4,2) {};
		\node[dot] (5) at (-5,0) {};
		\node[dot] (6) at (-4,-2) {};
		\draw[thick] (4) to (6);
		\draw[thick] (4) to (3);
		\draw[thick] (4) to (2);
		\node[frozen] (y1) at (-4.5,1) {$y_1$};
		\node[frozen] (y2) at (-4.5,-1) {$y_2$};
		\node[frozen] (y3) at (-2.5,-2) {$y_3$};
		\node[frozen] (y4) at (-0.5,-1) {$y_4$};
		\node[frozen] (y5) at (-0.5,1) {$y_5$};
		\node[frozen] (y6) at (-2.5,2) {$y_6$};
		\node[mutable] (x1) at (-4,0) {$x_1$};
		\node[mutable] (x2) at (-2.5,0) {$x_2$};
		\node[mutable] (x3) at (-1.6,0.8) {$x_3$};
	\end{tikzpicture}
	\hspace{0.75cm}
	\begin{tikzpicture}[baseline={(0,0)}]
		\node[frozen] (y1) at (-4,1) {$y_1$};
		\node[frozen] (y2) at (-4,-1) {$y_2$};
		\node[frozen] (y3) at (-2,-1) {$y_3$};
		\node[frozen] (y4) at (0,-1) {$y_4$};
		\node[frozen] (y5) at (2,1) {$y_6$};
		\node[frozen] (y6) at (2,-1) {$y_5$};
		\node[mutable] (x1) at (-3,0) {$x_1$};
		\node[mutable] (x2) at (-1,0) {$x_2$};
		\node[mutable] (x3) at (1,0) {$x_3$};
		
		\draw[-angle 90] (y1) to (x1);
		\draw[-angle 90] (x1) to (y2);
		\draw[-angle 90] (y3) to (x1);
		\draw[-angle 90] (x2) to (y3);
		\draw[-angle 90] (x1) to (x2);
		\draw[-angle 90] (y4) to (x2);
		\draw[-angle 90] (x2) to (x3);
		\draw[-angle 90] (x3) to (y4);
		\draw[-angle 90] (x3) to (y6);
		\draw[-angle 90] (y5) to (x3);
	\end{tikzpicture}
	\caption{An acyclic triangulation of a hexagon and its corresponding quiver.}
	\label{fig: polygon quiver}
\end{figure} 

\begin{rem}
	The same process can be used to generate a quiver for any triangulation of an unpunctured marked surface, possibly requiring one additional step. If there is a pair of arrows that connect the same two vertices with opposite orientations, remove that pair; repeat until there are no such pairs. For punctured surfaces, the reader should refer to Fomin, Shapiro, and Thurston \cite{FST08} for the construction of a quiver from a tagged triangulation.
\end{rem}

Cluster algebras of polygons are famously related to the family of cluster algebras of Type $A_n$ (as discussed in Section 5.3 of Fomin, Williams, and Zelevinsky's book \cite{FWZChapters45}), but not every quiver coming from a triangulation of a polygon is acyclic. We will say that a triangulation $T$ is \emph{acyclic} if the mutable part of the corresponding quiver $Q_T^{mut}$ is acyclic.

\begin{lemma}
	Let $T$ be a triangulation of a polygon $\Delta_m$. The mutable part of the quiver $Q_T^{mut}$ corresponding to $T$ is acyclic if and only if every triangle of $T$ contains at least one boundary arc. 
\end{lemma}
\begin{proof}
	Suppose $T$ contains a triangle $t$ with no boundary arcs. The three sides of $t$ correspond to mutable vertices which form a cycle of length 3 in $Q_T^{mut}$. Because $\Delta_m$ is a polygon, no two triangles share more than one side, so none of the three arrows are canceled by a cycle of length two, and $Q_T^{mut}$ is not acyclic.
	
	Now, suppose that every triangle of $T$ contains at least one boundary arc. Every triangle then contains at most one arrow in $Q_T^{mut}$, so there are no cycles within any triangle. The polygon is divided into $m-2$ triangles, so there are two triangles that contain two boundary arcs and one non-boundary arc; this non-boundary arc is then either a sink or source in $Q_T^{mut}$. The remaining triangles each contain two non-boundary arcs, 
    and those two arcs form a covering pair by Lemma \ref{lemma: Muller LA 10.1}, so they cannot be part of a cycle. 
    Thus, $Q_T^{mut}$ is acyclic.
\end{proof}

An \emph{anticlique} or \emph{independent set} is a subset of vertices of a graph such that such that the induced subgraph contains no edges. We now adapt several results of Lam and Speyer \cite{LS23} to the setting of polygons. Here, the cluster variables $x_1, \dotsc, x_{m-3}$ correspond to the mutable vertices, and $x_{m-3+1}, \dotsc, x_{2m-3}$ correspond to the frozen vertices.

\begin{lemma}\cite[Lemma 3.1]{LS23}\label{LSlemma31}
	Let $T$ be an acyclic triangulation of a polygon $\Delta_m$. For any point $x = (x_1, \dotsc, x_{2m-3})$ of $\mathcal{A}(\Delta_m; \Bbbk)$, the set of indices $I \subseteq \{1, \dotsc, m-3\}$ for which $x_i =0$ forms an anticlique. 
\end{lemma}

Let $\mathcal{I}_k$ denote the set of anticliques $I$ (as in Lemma \ref{LSlemma31}) of size $k$, and let $\mathcal{I}$ denote the set of all anticliques. For $I \in \mathcal{I}$, define $\mathcal{O}_I \subset \mathcal{A}(\Delta_m; \Bbbk)$ to be the relatively open set where $x_i =0$ for $i \in I$ and  $x_i \neq 0$ for $i \notin I$.
\begin{corollary}\cite[Corollary 3.2]{LS23}
	For an acyclic seed of a polygon $\Delta_m$, we have $\mathcal{A} = \bigsqcup_{I\in \mathcal{I}} \mathcal{O}_I$.
\end{corollary}

We can also describe these anticliques topologically.
\begin{lemma}
	Let $T$ be an acyclic triangulation of a polygon $\Delta_m$. By identifying the non-boundary arcs of $T$ with their indices in the corresponding quiver, \[ \mathcal{I} = \{\text{collections containing at most one non-boundary arc from each triangle}\}.\]
\end{lemma}
\begin{proof}
	This follows directly from the construction of a quiver from a triangulation.
\end{proof}

The cluster algebra of a triangulable polygon $\Delta_{m}$ is of Type $A_{m-3}$ with $m\geq 3$ frozen variables, so it is really full rank \cite[Section 5.1]{CGSS26}. Thus, we have the following proposition, adapted to the setting of polygons.
\begin{prop}\cite[Proposition 3.9]{LS23}\label{prop: LScount}
	Let $\mathcal{A}(\Delta_m; \Bbbk)$ be an acyclic cluster variety corresponding to a polygon $\Delta_m$. Then 
    \begin{equation}\label{eq: LScount} 
        \#\mathcal{A}(\Delta_m ;\mathbb{F}_q) = \sum_{k} ~\lvert \mathcal{I}_k \rvert q^k (q-1)^{2m-3-2k}.
    \end{equation}
\end{prop}

Our goal is to find a specialized formula for the number of points in the cluster variety of a marked surface, so we would like to simplify this expression as much as possible. Let us first examine the sizes of the sets of anticliques. We now let   
\[ I(k,n) \coloneq \lvert \mathcal{I}_k \rvert \text{ on the quiver of an acyclic triangulation of $\Delta_{n+3}$.} \]

$I(k,n)$ can then be defined recursively as follows:
\begin{align*}
	I(0,n) &= 1, \\
	I(1,n) &= n, \\
	I(k,n) &= I(k-1, n-2) + I(k,n-1).
\end{align*}

We can now rewrite Equation \ref{eq: LScount} using this recursion.
\begin{lemma}\label{lem: LS recurs}
	Suppose $m\geq 5$. Then 
	\begin{equation}\label{eq: LS recurs}
	    \#\mathcal{A}(\Delta_m ; \mathbb{F}_q) = (q-1)^2 \#\mathcal{A}(\Delta_{m-1} ; \mathbb{F}_q) + q(q-1)^2 \#\mathcal{A}(\Delta_{m-2} ; \mathbb{F}_q).
    \end{equation} 
\end{lemma}
\begin{proof} \begin{align*}
        &\#\mathcal{A}(\Delta_m ;\mathbb{F}_q) = \sum_{k=0}^{m-3} q^k (q-1)^{2m-3-2k} I(k,m-3) \\
		  &\quad= (q-1)^{2m-3} + (m-3)q(q-1)^{2m-5} + \sum_{k=2}^{m-3} q^k (q-1)^{2m-3-2k} I(k,m-3) \\
		&\quad= (q-1)^{2m-3} + (m-3)q(q-1)^{2m-5} \\ &\quad\quad+ \sum_{k=2}^{m-3} q^k (q-1)^{2m-3-2k} \big(I(k-1,m-5) + I(k,m-4) \big) \\
		&\quad= (q-1)^{2m-3} + q(q-1)^{2m-5} + (m-4)q(q-1)^{2m-5} \\ &\quad\quad+ q(q-1)^2 \sum_{k=2}^{m-3} q^{k-1} (q-1)^{2m-7-2(k-1)} I(k-1,m-5) \\ &\quad\quad+ (q-1)^2 \sum_{k=2}^{m-3} q^k (q-1)^{2m-5-2k} I(k,m-4) \\
		&\quad= q(q-1)^2 \bigg((q-1)^{2m-7} + \sum_{k=1}^{m-3} q^{k} (q-1)^{2m-7-2k} I(k,m-5) \bigg) \\ &\quad\quad+ (q-1)^2 \bigg( (q-1)^{2m-5} + (m-4)q(q-1)^{2m-7} + \sum_{k=2}^{m-3} q^k (q-1)^{2m-5-2k} I(k,m-4) \bigg) \\
		&\quad= q(q-1)^2 \#\mathcal{A}(\Delta_{m-2} ;\mathbb{F}_q) + (q-1)^2 \#\mathcal{A}(\Delta_{m-1} ;\mathbb{F}_q). \qedhere
	\end{align*}    
\end{proof}

\subsection{Counting points}

We first compute the values of $\#\mathcal{A}(\Delta_3 ; \mathbb{F}_q)$ and $\#\mathcal{A}(\Delta_4 ; \mathbb{F}_q)$ using Equation \ref{eq: LScount} to provide the base cases for computing a general formula.

\begin{itemize}
\item Case $m=3$: The trigon has no mutable arcs, so the only nonzero value of $I(k,m-3)$ is $I(0,0) = 1$. The number of points over $\mathbb{F}_q$ is then 
\begin{align*}
	\#\mathcal{A}(\Delta_3 ;\mathbb{F}_q) &= \sum_{k=0}^{m-3} q^k (q-1)^{2m-3-2k} I(k,m-3) \\
	&= (q-1)^3 \\
	&= (q-1)^{m-1} \Bigg( \frac{q^{m-1} + (-1)^m}{q+1} \Bigg).
\end{align*}
\item Case $m=4$: The quadrilateral has one mutable arc, so the only nonzero values of $I(k,m-3)$ are $ I(0,1) = 1 $ and $ I(1,1) = 1 $. The number of points over $\mathbb{F}_q$ is then 
\begin{align*}
	\#\mathcal{A}(\Delta_4 ;\mathbb{F}_q) &= \sum_{k=0}^{m-3} q^k (q-1)^{2m-3-2k} I(k,m-3) \\
	&= (q-1)^5 + q(q-1)^3 \\
	&= (q-1)^3 (q^2 -q +1) \\
	&= (q-1)^{m-1} \Bigg( \frac{q^{m-1} + (-1)^m}{q+1} \Bigg).
\end{align*}
\end{itemize}

\begin{prop}\label{prop:polygon points}
	Given a polygon $\Delta_m$ with $m\geq 3$ sides, the corresponding cluster variety over $\mathbb{F}_q$ contains 
    \begin{equation}\label{eq: point count polygon}
        \#\mathcal{A}(\Delta_m; \mathbb{F}_q) = (q-1)^{m-1} \Bigg( \frac{q^{m-1} + (-1)^{m}}{q+1} \Bigg) 
    \end{equation} 
	points.
\end{prop}
\begin{proof}
	The base cases are calculated above. Now for our inductive step, assume there is some integer $k$ such that the theorem holds for all $3 \leq m \leq k-1$. Using Lemma \ref{lem: LS recurs}, 
	\begin{align*}
		\#\mathcal{A}(\Delta_k ; \mathbb{F}_q) &= (q-1)^2 \#\mathcal{A}(\Delta_{k-1} ; \mathbb{F}_q) + q(q-1)^2 \#\mathcal{A}(\Delta_{k-2} ; \mathbb{F}_q) \\
		&= (q-1)^{k} \Bigg( \frac{q^{k-2} + (-1)^{k-1}}{q+1} \Bigg) + q(q-1)^{k-1} \Bigg( \frac{q^{k-3} + (-1)^{k-2}}{q+1} \Bigg) \\
		&= \frac{(q-1)^{k-1}}{q+1} \big( (q-1)\big(q^{k-2} + (-1)^{k-1}\big) + q \big(q^{k-3} + (-1)^{k-2}\big)\big) \\
		&= \frac{(q-1)^{k-1}}{q+1} \big( q^{k-1} + (-1)^{k} \big).
	\end{align*}
	Thus, the formula holds for all integers $k\geq 3$. 
\end{proof}

\begin{rem}
    Note that Proposition \ref{prop:polygon points} agrees with Chapoton's formula \cite[Lemma 3.3]{Chap11}, in the sense that we have scaled that formula by $(q-1)^{m-1}$ to account for all other possible boundary values.
\end{rem}

\section{Cluster varieties of unpunctured surfaces}\label{section: unpunctured}

While our goal is a general formula for the number of points in the cluster variety associated to a general marked surface with boundary, we will begin with genus 0 surfaces and build a more general formula incrementally. 

\subsection{Annuli}

As noted in Example \ref{ex: annulus cutting}, cutting the annulus at either arc of the covering pair in Figure \ref{fig: covering pair} produces a polygon with $m+2$ sides. The cluster varieties corresponding to these polygons cover the cluster variety corresponding to $\Sigma_{0,2,m}$. We need only count the intersection of these varieties. 

One minor complication is that each polygon resulting from cutting contains two copies of the cut arc, so those arcs must take the same value. The cluster variety corresponding to one of these cut polygons \textemdash which we will denote as $\Delta_m/\!\sim $ \textemdash has one fewer choice of value for the boundary arcs, resulting in the following. 

\begin{lemma}
Given a polygon $\Delta_m/\!\sim$ with $m\geq 3$ sides that is the result of cutting, the corresponding cluster variety over $\mathbb{F}_q$ contains 
\begin{equation}
    \#\mathcal{A}(\Delta_m/\!\sim; \mathbb{F}_q) = (q-1)^{m-2} \Bigg( \frac{q^{m-1} + (-1)^{m}}{q+1} \Bigg) 
\end{equation} 
points.
\end{lemma}

The intersection consists of points contained in the varieties corresponding to both cut polygons, so the intersection is precisely the points that are nonzero on both arcs of the covering pair. Consider the polygon on the right side of Figure \ref{fig: cuttings}, where $\alpha$ is  a boundary arc so it must take a nonzero value. In the intersection, $\beta$ must also be nonzero, so we can treat it as a boundary of an $(m+1)$-gon. The variety corresponding to that smaller polygon will be contained in the intersection of the varieties. Considering the possible values of the boundary arcs of the original $(m+2)$-gon \textemdash namely $\alpha$ and $\gamma$ \textemdash we see that $\alpha$ must take the same value as the other copy of the cut arc, but there are $q-1$ possible values of $\gamma$, giving us $q-1$ copies of this $(m+1)$-gon in the intersection of the two varieties.

\begin{prop}\label{prop:annulus points}
    Given an annulus $\Sigma_{0,2,m}$ with at least one marked point on each boundary component, the corresponding cluster variety over $\mathbb{F}_q$ contains 
    \begin{equation} 
        \#\mathcal{A}(\Sigma_{0,2,m}; \mathbb{F}_q) = (q-1)^{m} \big( q^{m} + (-1)^{m} \big) 
    \end{equation} 
    points.
\end{prop}
\begin{proof}
	We add two copies of the variety of the $(m+2)$-gon $\Delta_{m+2}/\!\sim$ and subtract $(q-1)$ copies of the variety of the $(m+1)$-gon $\Delta_{m+1}$.
    \begin{align*}
        \#\mathcal{A}(\Sigma_{0,2,m}; \mathbb{F}_q) &= 2 (q-1)^{m} \Bigg( \frac{q^{m+1} + (-1)^{m}}{q+1} \Bigg) - (q-1)(q-1)^{m} \Bigg( \frac{q^{m} + (-1)^{m-1}}{q+1} \Bigg) \\
        &= \frac{(q-1)^{m}}{q+1} \big( 2\big(q^{m+1} + (-1)^{m}\big) - (q-1)\big( q^{m} + (-1)^{m-1}\big) \big) \\
        &= \frac{(q-1)^{m}}{q+1} \big( q^{m+1} +2(-1)^{m} +q^m +q(-1)^m - (-1)^m  \big) \\
        &= \frac{(q-1)^{m}}{q+1} \big( (q+1)q^{m} + (q+1)(-1)^{m} \big) \\
        &= (q-1)^{m} \big( q^{m} + (-1)^{m} \big). \qedhere
    \end{align*}
\end{proof}

\begin{rem}
    In the case $q=2$, the above formula simplifies to $2^m + (-1)^m$, which is the $m$-th \emph{Jacobsthal-Lucas number} (Sequence \seqnum{A014551} in the OEIS \cite{oeis}). As a result, the number of points in a cluster variety corresponding to a marked annulus over $\mathbb{F}_q$ can be regarded as a generalization of the Jacobsthal-Lucas numbers. 
\end{rem}

\subsection{Other genus 0 surfaces}

More generally, given a surface $\Sigma_{0,b,m}$ of genus 0 with $b\geq 2$ boundary components and $m$ total marked points, we can find a covering pair in exactly the same way as with the annulus \textemdash a pair of arcs with end points on the boundary that cut out an unpunctured triangle. The cluster variety is then covered by two surfaces of genus 0 with $b-1$ boundary components. A formula can now be recursively defined as before: 
\begin{equation}
    \#\mathcal{A}(\Sigma_{0,b,m}; \mathbb{F}_q) =  2\#\mathcal{A}(\Sigma_{0,b-1,m+2}/\!\sim; \mathbb{F}_q) - (q-1) \#\mathcal{A}(\Sigma_{0,b-1,m+1}; \mathbb{F}_q).
\end{equation}
\begin{prop}\label{prop:genus 0 points}
    Let $\Sigma_{0,b,m}$ be a genus 0 surface with $b\geq 1$ boundary components and $m$ total marked points, where either 
    \begin{itemize}
        \item $b = 1$ and $m\geq 3$, or
        \item $b \geq 2$ and each boundary component contains at least one marked point.
    \end{itemize}
    The corresponding cluster variety over $\mathbb{F}_q$ contains 
    \begin{equation} 
        \#\mathcal{A}(\Sigma_{0,b,m}; \mathbb{F}_q) = (q+1)^{b-2} (q-1)^{m+b-2} \big( q^{m+b-2} + (-1)^{m}\big) 
    \end{equation} 
    points.
\end{prop}
\begin{proof}
    First, consider the case of $b=1$ and $m\geq 3$. From Proposition \ref{prop:polygon points}, we have 
    \begin{align*}
     \#\mathcal{A}(\Sigma_{0,1,m}; \mathbb{F}_q) &= (q-1)^{m-1} \Bigg( \frac{q^{m-1} + (-1)^{m}}{q+1} \Bigg) \\
     &= (q+1)^{b-2} (q-1)^{m+b-2} \big( q^{m+b-2} + (-1)^{m}\big).
    \end{align*}
    
    Second, consider the case of $b=2$, where each boundary component contains at least one marked point. From Proposition \ref{prop:annulus points}, we have
    \begin{align*}
     \#\mathcal{A}(\Sigma_{0,2,m}; \mathbb{F}_q) &= (q-1)^{m} \big( q^{m} + (-1)^{m} \big) \\
     &= (q+1)^{b-2} (q-1)^{m+b-2} \big( q^{m+b-2} + (-1)^{m}\big).
    \end{align*}
    
    For our inductive hypothesis, suppose the theorem holds for all $1 \leq b<k$ for some $k$. Using our recursive formula, we compute the number of points in the cluster variety corresponding to $\Sigma_{0,k,m}$ to be 
    \begin{align*}
        \#\mathcal{A}(\Sigma_{0,k,m}; \mathbb{F}_q) &= 2 \#\mathcal{A}(\Sigma_{0,k-1,m+2}/\!\sim; \mathbb{F}_q) - (q-1)\#\mathcal{A}(\Sigma_{0,k-1,m+1}; \mathbb{F}_q) \\
        &= 2(q+1)^{k-3} (q-1)^{m+k-2} \big( q^{m+k-1} + (-1)^{m+2}\big) \\ &\quad - (q+1)^{k-3} (q-1)^{m+k-1} \big( q^{m+k-2} + (-1)^{m+1}\big) \\
        &= (q+1)^{k-3} (q-1)^{m+k-2} \big( 2\big(q^{m+k-1} + (-1)^{m+2}\big) \\ &\quad - (q-1)\big(q^{m+k-2} + (-1)^{m+1}\big)\big) \\
        &= (q+1)^{k-3} (q-1)^{m+k-2} \big( q^{m+k-1} + (-1)^{m+2} - q(-1)^{m+1} + q^{m+k-2} \big) \\
        &= (q+1)^{k-2} (q-1)^{m+k-2}\big( q^{m+k-2} + (-1)^{m}\big).
    \end{align*}
    Thus, the formula holds for all $b \geq 1$. 
\end{proof}

\subsection{Positive genus surfaces}

\begin{figure}[ht]
    \centering
    \begin{tikzpicture}[scale=1.0]
        \draw[thick] (-3,-1) to [out=180,in=270] (-3.25,-0.6) to (-3.25,0.6) to [out=90,in=180] (-3,1) to [out=0,in=90] (-2.75,0.6) to (-2.75,-0.6) to [out=270,in=0] (-3,-1) {};
        \node[dot] (1) at (-2.75,0.5) {};
        \node[dot] (2) at (-2.75,-0.5) {};
        \draw[thick] (-3,1) to [out=0,in=180] (-2,1.5) to (2,1.5) {};
        \draw[thick] (-3,-1) to [out=0,in=180] (-2,-1.5) to (2,-1.5) {};
        \draw[dashed] (2,1.5) to (2,-1.5) {};
        \draw[thick] (-1.25,0) to [out=345,in=195] (0.25,0) to [out=150,in=30] (-1.25,0) {}; 
        \draw[thick] (-1.25,0) to [out=165,in=345] (-1.4,0.05) {};
        \draw[thick] (0.25,0) to [out=15,in=195] (0.4,0.05) {};
        \draw[thick] (1) to [out=10,in=170] (-0.1,0.6) to [out=350,in=90] (0.6,0) to [out=270,in=15] (0.1,-0.5) to [out=195,in=330] (1) {};      
        \draw[thick] (1) to [out=25,in=175] (0.3,1.0) to [out=355,in=90] (1.4,0) to [out=270,in=5] (0.3,-1.0) to [out=185,in=345] (2) {};   
        
        \node at (0.79,0) {$\alpha$};
        \node at (1.58,0) {$\beta$};
    \end{tikzpicture}
    \caption{A covering pair $(\alpha,\beta)$ in a marked surface with boundary and genus.}
    \label{fig: covering pair genus}
\end{figure}
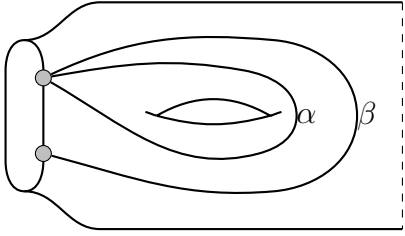

Suppose $g\geq 1$. If there is a boundary component containing at least two marked points, the pair of arcs $(\alpha,\beta)$ pictured in Figure \ref{fig: covering pair genus} is a covering pair by Lemma \ref{lemma: Muller LA 10.1}. We can cut along this covering pair, giving us the recursive formula \[ \#\mathcal{A}(\Sigma_{g,b,m}; \mathbb{F}_q) =  2\#\mathcal{A}(\Sigma_{g-1,b+1,m+2}/\!\sim; \mathbb{F}_q) - (q-1) \#\mathcal{A}(\Sigma_{g-1,b+1,m+1}; \mathbb{F}_q).\]

If $g=1$, we have 
\begin{align*}
    \#\mathcal{A}(\Sigma_{1,b,m}; \mathbb{F}_q) &=  2\#\mathcal{A}(\Sigma_{0,b+1,m+2}/\!\sim; \mathbb{F}_q) - (q-1) \#\mathcal{A}(\Sigma_{0,b+1,m+1}; \mathbb{F}_q) \\
    &= 2 (q+1)^{b-1} (q-1)^{m+b} \big(q^{m+b+1} + (-1)^{m+2}\big) \\ &\quad - (q+1)^{b-1} (q-1)^{m+b+1} \big(q^{m+b} + (-1)^{m+1}\big) \\
    &= (q+1)^{b-1} (q-1)^{m+b} \big(q^{m+b+1} + (-1)^{m+2} + q^{m+b} + q(-1)^{m+2} \big) \\
    &= (q+1)^{b-1} (q-1)^{m+b} \big(q^{m+b}(q+1) + (-1)^{m+2}(q+1) \big) \\
    &= (q+1)^{b} (q-1)^{m+b} \big(q^{m+b} + (-1)^{m+2} \big). 
\end{align*}

Using that result, if $g=2$, we find 
\begin{align*}
    \#\mathcal{A}(\Sigma_{2,b,m}; \mathbb{F}_q) &=  2\#\mathcal{A}(\Sigma_{1,b+1,m+2}/\!\sim; \mathbb{F}_q) - (q-1) \#\mathcal{A}(\Sigma_{1,b+1,m+1}; \mathbb{F}_q) \\
    &= 2 (q+1)^{b+1} (q-1)^{m+b+2} \big(q^{m+b+3} + (-1)^{m+4} \big) \\ &\quad - (q+1)^{b+1} (q-1)^{m+b+3} \big(q^{m+b+2} + (-1)^{m+3} \big) \\
    &= (q+1)^{b+1} (q-1)^{m+b+2} \big(q^{m+b+3} + (-1)^{m+4} + q^{m+b+2} + q(-1)^{m+4} \big) \\
    &= (q+1)^{b+1} (q-1)^{m+b+2} \big(q^{m+b+2}(q+1) + (-1)^{m+4}(q+1) \big) \\
    &= (q+1)^{b+2} (q-1)^{m+b+2} \big(q^{m+b+2} + (-1)^{m+4} \big). 
\end{align*}

\begin{prop}\label{prop:genus g points}
    Let $\Sigma_{g,b,m}$ be a genus $g\geq 0$ surface with $b\geq 1$ boundary components and $m$ total marked points, where either 
    \begin{itemize}
        \item $g =0$, $b = 1$, and $m\geq 3$, 
        \item $g=0$, $b \geq 2$, and each boundary component contains at least one marked point, or
        \item $g \geq 1$, $b \geq 1$, $m\geq 2$, each boundary component contains at least one marked point, and at least two marked points lie on the same boundary component.
    \end{itemize}
    The corresponding cluster variety over $\mathbb{F}_q$ contains 
    \begin{equation} 
        \#\mathcal{A}(\Sigma_{g,b,m}; \mathbb{F}_q) = (q+1)^{2g+b-2} (q-1)^{2g+m+b-2} \big(q^{2g+m+b-2} + (-1)^{m}\big) 
    \end{equation}
    points.
\end{prop}
\begin{proof}
    The base cases are above. We need only do induction on genus. Suppose that the theorem holds for all $0 \leq g < k$ for some $k>2$. We then compute the number of points in the cluster variety corresponding to $\Sigma_{k,b,m}$ to be
    \begin{align*}
        &\#\mathcal{A}(\Sigma_{k,b,m}; \mathbb{F}_q) =  2\#\mathcal{A}(\Sigma_{k-1,b+1,m+2}/\!\sim; \mathbb{F}_q) - (q-1) \#\mathcal{A}(\Sigma_{k-1,b+1,m+1}; \mathbb{F}_q) \\
        &\quad = 2(q+1)^{2k+b-3} (q-1)^{2k+m+b-2} \big(q^{2k+m+b-1} + (-1)^{m+2}\big) \\ &\quad\quad - (q+1)^{2k+b-3} (q-1)^{2k+m+b-1} \big(q^{2k+m+b-2} + (-1)^{m+1}\big) \\
        &\quad= (q+1)^{2k+b-3} (q-1)^{2k+m+b-2} \big( q^{2k+m+b-1} + (-1)^{m} - q(-1)^{m+1} +q^{2k+m+b-2} \big) \\
        &\quad= (q+1)^{2k+b-3} (q-1)^{2k+m+b-2} \big( (q+1)q^{2k+m+b-2} + (q+1)(-1)^{m} \big) \\
        &\quad= (q+1)^{2k+b-2} (q-1)^{2k+m+b-2} \big( q^{2k+m+b-2} + (-1)^{m} \big).
    \end{align*}
    Thus, the formula holds for all $g\geq 0$. 
\end{proof}

\section{Cluster varieties of punctured surfaces}\label{section: punctured surfaces}

We now begin to examine cluster varieties associated to punctured surfaces. If at least two marked points are contained in the same boundary component, we can use the covering pair from Figure \ref{fig: covering pair punctured}. Cutting along $\alpha$ disconnects the surface, giving us $\Sigma_{g,b,m}^{p-1} \bigsqcup \Sigma_{0,1,2}^1 /\!\sim$ with the two copies of $\alpha$ required to take the same value. Cutting along $\beta$ give us a surface with one fewer puncture $\Sigma_{g,b,m+2}^{p-1}/\!\sim$ where both copies of $\beta$ are required to take the same value.
We already have a formula for unpunctured surfaces, so we can attempt to find a recursive formula using that result, but first we will need to calculate the number of points in the cluster variety of the punctured digon. 

\subsection{Once-punctured digon}

Figure \ref{fig: punctured digon quiver} displays one quiver of the punctured digon. The mutable part of the quiver is acyclic, so we have a presentation of the cluster algebra due to Berenstein, Fomin, and Zelevinsky \cite[Corollary 1.21]{BFZ05}, stated more precisely for our setting by Lam and Speyer \cite[Theorem 2.4]{LS23}. The corresponding cluster algebra can then be presented as \[ A(\Sigma_{0,1,2}^1) \cong \langle x_1, x_1^\prime, x_2, x_2^\prime, y_1^{\pm 1}, y_2^{\pm 1} \;|\; x_1 x_1^\prime = x_2 x_2^\prime = y_1 + y_2 \rangle. \]

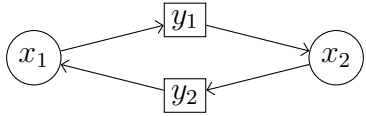
\begin{figure}[ht]
    \centering
    \begin{tikzpicture}
        \node[mutable] (x1) at (-2,0) {$x_1$};
        \node[mutable] (x2) at (2,0) {$x_2$};
        \node[frozen] (y1) at (0,0.5) {$y_1$};
        \node[frozen] (y2) at (0,-0.5) {$y_2$};
        
        \draw[-angle 90] (x1) to (y1);
        \draw[-angle 90] (y1) to (x2);
        \draw[-angle 90] (x2) to (y2);
        \draw[-angle 90] (y2) to (x1);
    \end{tikzpicture}
    \caption{A quiver of a punctured digon.}
    \label{fig: punctured digon quiver}
\end{figure} 

\begin{prop}\label{prop: digon count}
    The cluster variety of the punctured digon over $\mathbb{F}_q$ has 
    \begin{equation}\label{eq: digon count} 
        \#\mathcal{A}(\Sigma_{0,1,2}^1; \mathbb{F}_q) = q^4 - q^3 + q^2 -2q +1 = (q-1)(q^3 + q - 1) 
    \end{equation} 
    points.
\end{prop}
\begin{proof}
    We can examine the points directly. If $y_1 + y_2 = 0$, then there are points of the following types:
        \begin{itemize}
            \item $(q-1)$ points of the form $(0,0,0,0,a,-a)$ for some $a \in \mathbb{F}_q^\times$, 
            \item $(q-1)^2$ points of each of the forms $(b,0,0,0,a,-a)$, $(0,b,0,0,a,-a)$, $(0,0,b,0,a,-a)$, and $(0,0,0,b,a,-a)$ for some $a,b \in \mathbb{F}_q^\times$, and 
            \item $(q-1)^3$ points of each of the forms $(b,0,c,0,a,-a)$, $(0,b,c,0,a,-a)$, $(b,0,0,c,a,-a)$, and $(0,b,0,c,a,-a)$ for some $a,b,c \in \mathbb{F}_q^\times$.
        \end{itemize}
        
    If $y_1 + y_2 \neq 0$, then there are $q-1$ possible values for $y_1$, and after choosing a value for $y_1$, there are $q-2$ possible values for $y_2$. This gives us 
    \begin{itemize}
        \item $(q-2)(q-1)^3$ points of the form $(c,\frac{a+b}{c}, d, \frac{a+b}{d}, a, b)$ for some $a,b,c,d \in \mathbb{F}_q^\times$ such that $a+b \neq 0$.
    \end{itemize}
    
    Adding these together, we get the following formula:
    \begin{align*}
        \#\mathcal{A}(\Sigma_{0,1,2}^1; \mathbb{F}_q) &= (q-1) + 4(q-1)^2 +4(q-1)^3 + (q-2)(q-1)^3\\
        &= q^4 - q^3 + q^2 -2q +1 \\
        &= (q-1)(q^3 +q - 1). \qedhere
    \end{align*}
\end{proof}

\subsection{Once-punctured surfaces}

Let us begin with $p=1$. Using the covering pair from Figure \ref{fig: covering pair punctured}, cutting along $\alpha$ results in the disconnected surface $(\Sigma_{g,b,m}^{0} \bigsqcup \Sigma_{0,1,2}^1 )/\!\sim$, and each point of $\mathcal{A}((\Sigma_{g,b,m}^{0} \bigsqcup \Sigma_{0,1,2}^1 )/\!\sim; \mathbb{F}_q)$ is equivalent to a pair of points $ p_1 \in \mathcal{A}(\Sigma_{g,b,m}^{0}; \mathbb{F}_q)$ and $p_2 \in \mathcal{A}(\Sigma_{0,1,2}^1; \mathbb{F}_q)$ such that $p_1$ and $p_2$ agree on the separate copies of $\alpha$. That is, 
\[\#\mathcal{A}((\Sigma_{g,b,m}^{0} \bigsqcup \Sigma_{0,1,2}^1 )/\!\sim; \mathbb{F}_q) \simeq (\mathcal{A}(\Sigma_{g,b,m}^{0}; \mathbb{F}_q) \times \mathcal{A}(\Sigma_{0,1,2}^1; \mathbb{F}_q))/\!\sim.\] 
Because we only choose the value for $\alpha$ once, the $(q-1)$ factor from the digon cancels, and we have 
\begin{equation} 
    \#\mathcal{A}((\Sigma_{g,b,m}^{0} \bigsqcup \Sigma_{0,1,2}^1) /\!\sim; \mathbb{F}_q) = (q+1)^{2g+b-2} (q-1)^{2g+m+b-2} \big(q^{2g+m+b-2} + (-1)^{m}\big) (q^3 + q - 1). 
\end{equation} 

Cutting along $\beta$ gives us $\Sigma_{g,b,m+2}^0$, and we have 
\begin{equation} 
    \#\mathcal{A}(\Sigma_{g,b,m+2}^{0} /\!\sim; \mathbb{F}_q) = (q+1)^{2g+b-2} (q-1)^{2g+m+b-1} \big(q^{2g+m+b} + (-1)^{m}\big) 
\end{equation} 
points that agree on the two copies of $\beta$. 

We now need to determine the size of the intersection of these two varieties. For any point in the intersection, both $\alpha$ and $\beta$ must be nonzero. From the proof of Proposition \ref{prop: digon count}, we see that any given radius of the digon takes a nonzero value on 
\[ (q-1)^2 + 2(q-1)^3 + (q-2)(q-1)^3 = (q-1)^2(q^2 - q + 1) \]
of the points in $\mathcal{A} (\Sigma_{0,1,2}^1; \mathbb{F}_q)$. The copies of $\alpha$ must have the same value, so we divide by $q-1$ and then take the product with the unpunctured surface, finding that $\beta$ is nonzero on 
\[ (q+1)^{2g+b-2} (q-1)^{2g+m+b-1} \big(q^{2g+m+b-2} + (-1)^{m}\big) (q^2 - q + 1) \] 
points of $ \#\mathcal{A}(\Sigma_{g,b,m}^{0} \bigsqcup \Sigma_{0,1,2}^1 /\!\sim; \mathbb{F}_q) $.

To check this count, we now consider the surface created by cutting along $\beta$. The points that are nonzero on $\alpha$ correspond to a cutting along $\alpha$, giving us $(\Sigma_{g,b,m}^0 \bigsqcup (\Sigma_{0,1,4}^0/\!\sim))/\!\sim$. The quadrilateral has two copies of $\beta$ as adjacent sides, and the value of one of the other sides is determined as it is a copy of $\alpha$, so we will divide by $(q-1)^2$. Taking the product, we have 
\begin{align*}
 &\#\mathcal{A}((\Sigma_{g,b,m}^0 \bigsqcup (\Sigma_{0,1,4}^0/\!\sim))/\!\sim; \mathbb{F}_q) \\
 &\quad= (q+1)^{2g+b-2} (q-1)^{2g+m+b-2} \big(q^{2g+m+b-2} + (-1)^{m}\big) \frac{(q-1)^3}{(q-1)^2} \bigg( \frac{q^3 + 1}{q+1} \bigg) \\
 &\quad= (q+1)^{2g+b-2} (q-1)^{2g+m+b-1} \big(q^{2g+m+b-2} + (-1)^{m}\big)  \big( q^2 - q + 1 \big).
\end{align*}
This is the same number of points as in the other cutting, so we have verified that this is the number of points in the intersection of the cut varieties.

Adding the number of points in the two cut varieties and subtracting the points from the intersection, we get 
\begin{align*}
    \#\mathcal{A}(\Sigma_{g,b,m}^1; \mathbb{F}_q) &= (q+1)^{2g+b-2}(q-1)^{2g+m+b-2}\big(q^{2g+m+b-2} + (-1)^{m}\big)(q^3+q-1) \\
    &\quad + (q+1)^{2g+b-2}(q-1)^{2g+m+b-1}\big(q^{2g+m+b} + (-1)^{m}\big) \\
    &\quad - (q+1)^{2g+b-2}(q-1)^{2g+m+b-1}\big(q^{2g+m+b-2} + (-1)^{m}\big)(q^2-q+1).
\end{align*}
Rewriting the first and third terms as multiples of $\#\mathcal{A}(\Sigma_{g,b,m}^0; \mathbb{F}_q)$ gives us 
\begin{align*}
    \#\mathcal{A}(\Sigma_{g,b,m}^1; \mathbb{F}_q) &= (q^3+q-1) \#\mathcal{A}(\Sigma_{g,b,m}^0; \mathbb{F}_q) + \#\mathcal{A}(\Sigma_{g,b,m+2}^0/\!\sim; \mathbb{F}_q) \\ &\quad - (q-1)(q^2-q+1) \#\mathcal{A}(\Sigma_{g,b,m}^0; \mathbb{F}_q) \\
    &= (2q^2-q) \#\mathcal{A}(\Sigma_{g,b,m}^0; \mathbb{F}_q) + \#\mathcal{A}(\Sigma_{g,b,m+2}^0/\!\sim; \mathbb{F}_q) \\
    &= (q+1)^{2g+b-2} (q-1)^{2g+m+b-2} \big((2q^2-q)\big(q^{2g+m+b-2} + (-1)^{m}\big) \\ &\quad  + (q-1)\big(q^{2g+m+b} + (-1)^{m}\big) \big) \\
    &= (q+1)^{2g+b-2} (q-1)^{2g+m+b-2} \big((q^2 + q - 1)q^{2g+m+b-1} + (2q^2-1)(-1)^{m} \big).
\end{align*}

\subsection{Surfaces with multiple punctures}

A surface with multiple punctures can be cut in exactly the same manner. This gives us a recursive formula that holds for all punctured surfaces:
\begin{align*}
    \#\mathcal{A}(\Sigma_{g,b,m}^p; \mathbb{F}_q) &= (q^3+q-1) \#\mathcal{A}(\Sigma_{g,b,m}^{p-1}; \mathbb{F}_q) + \#\mathcal{A}(\Sigma_{g,b,m+2}^{p-1}/\!\sim; \mathbb{F}_q) \\ &\quad - (q-1)(q^2-q+1) \#\mathcal{A}(\Sigma_{g,b,m}^{p-1}; \mathbb{F}_q) \\
    &= (2q^2-q) \#\mathcal{A}(\Sigma_{g,b,m}^{p-1}; \mathbb{F}_q) + \#\mathcal{A}(\Sigma_{g,b,m+2}^{p-1}/\!\sim; \mathbb{F}_q).
\end{align*}

Let us calculate the formula for $p=2$:
 \begin{align*}
    &\#\mathcal{A}(\Sigma_{g,b,m}^p; \mathbb{F}_q) = (2q^2-q) \#\mathcal{A}(\Sigma_{g,b,m}^{p-1}; \mathbb{F}_q) + \#\mathcal{A}(\Sigma_{g,b,m+2}^{p-1}/\!\sim; \mathbb{F}_q) \\
    &\quad= (2q^2-q)(q+1)^{2g+b-2} (q-1)^{2g+m+b-2} \big((q^2 + q - 1)q^{2g+m+b-1} + (2q^2-1)(-1)^{m} \big) \\
    &\quad\quad+ (q+1)^{2g+b-2} (q-1)^{2g+m+b-1} \big((q^2 + q - 1)q^{2g+m+b+1} + (2q^2-1)(-1)^{m} \big) \\
    &\quad= (q+1)^{2g+b-2} (q-1)^{2g+m+b-2} \big(\big(2q^5+q^4-3q^3+q^2\big)q^{2g+m+b-2}  \\ &\quad\quad  + (2q^2-q)(2q^2-1)(-1)^m + \big(q^6-2q^4+q^3\big)q^{2g+m+b-2} + (q-1)(2q^2-1)(-1)^m \big) \\
    &\quad= (q+1)^{2g+b-2} (q-1)^{2g+m+b-2} \big(\big(q^6 + 2q^5-q^4-2q^3+q^2\big)q^{2g+m+b-2}  \\ &\quad\quad + (2q^2-1)^2(-1)^m  \big) \\
    &\quad= (q+1)^{2g+b-2} (q-1)^{2g+m+b-2} \big((q^2+q-1)^2 q^{2g+m+b} + (2q^2-1)^2(-1)^m \big). 
\end{align*}

Now, we show that the formula extends to surfaces with any number of punctures.
\begin{prop}\label{prop: punctured count}
    Let $\Sigma_{g,b,m}^p$ be a genus $g$ surface with $b\geq1$ boundary components, $m\geq b$ marked points, and $p\geq 0$ punctures, where one of the following holds:
    \begin{itemize}
        \item $g=0$, $p\geq 0$, $b=1$, and $m\geq 3$, 
        \item $g=0$, $p=0$, $b\geq 2$, and each boundary component contains at least one marked point,
        \item $g \geq 1$, $p\geq 0$, $b \geq 1$, $m\geq 2$, each boundary component contains at least one marked point, and at least two marked points lie on the same boundary component, or
        \item $g \geq 0$, $p \geq 1$, $b \geq 1$, $m\geq 2$, each boundary component contains at least one marked point, and at least two marked points lie on the same boundary component. 
    \end{itemize}
    The corresponding cluster variety over $\mathbb{F}_q$ contains 
    \begin{equation}\label{eq: prop punctured count}
        \#\mathcal{A}(\Sigma_{g,b,m}^p; \mathbb{F}_q) = (q+1)^{2g+b-2} (q-1)^{2g+m+b-2} \big( q^{2g+m+b+p-2} (q^2+q-1)^p + (-1)^m (2q^2-1)^p \big) 
    \end{equation}
    points.
\end{prop}
\begin{proof}
    The base cases are above. We need only do induction on the number of punctures. Suppose there is an integer $k$ such that the formula holds for all $0 \leq p \leq k-1$. 
    \begin{align*}
        &\#\mathcal{A}(\Sigma_{g,b,m}^p; \mathbb{F}_q) = (2q^2-q) \#\mathcal{A}(\Sigma_{g,b,m}^{p-1}; \mathbb{F}_q) + \#\mathcal{A}(\Sigma_{g,b,m+2}^{p-1}/\!\sim; \mathbb{F}_q) \\
        &\quad= (2q^2 - q)(q+1)^{2g+b-2} (q-1)^{2g+m+b-2} \big( q^{2g+m+b+p-3} (q^2+q-1)^{p-1}  \\ &\quad\quad + (-1)^m (2q^2-1)^{p-1} \big) + (q+1)^{2g+b-2} (q-1)^{2g+m+b-1} \big( q^{2g+m+b+p-1} (q^2+q-1)^{p-1}  \\ &\quad\quad  + (-1)^{m+2} (2q^2-1)^{p-1} \big) \\
        &\quad= (q+1)^{2g+b-2} (q-1)^{2g+m+b-2} \big( (2q^2 - q)\big(q^{2g+m+b+p-3} (q^2+q-1)^{p-1}  \\ &\quad\quad  + (-1)^m (2q^2-1)^{p-1}\big) + (q-1) \big(q^{2g+m+b+p-1} (q^2+q-1)^{p-1}  \\ &\quad\quad   + (-1)^m (2q^2-1)^{p-1} \big)\big) \\
        &\quad= (q+1)^{2g+b-2} (q-1)^{2g+m+b-2} \big( 2q^{2g+m+b+p-1} (q^2+q-1)^{p-1} + 2q^2(-1)^m (2q^2-1)^{p-1}  \\ &\quad\quad - q^{2g+m+b+p-2} (q^2+q-1)^{p-1} -q (-1)^m (2q^2-1)^{p-1} \\ &\quad\quad + q^{2g+m+b+p} (q^2+q-1)^{p-1} + q(-1)^m (2q^2-1)^{p-1} \\ &\quad\quad  - q^{2g+m+b+p-1} (q^2+q-1)^{p-1} - (-1)^m (2q^2-1)^{p-1} \big) \\
        &\quad= (q+1)^{2g+b-2} (q-1)^{2g+m+b-2} \big( q^{2g+m+b+p-2} (q^2+q-1)^{p} + (-1)^m (2q^2-1)^{p}\big).\qedhere
    \end{align*}
\end{proof}

\subsection{Surfaces with one marked point on each boundary component}

\begin{figure}[htb]
	\centering
	\begin{tikzpicture}[scale=0.9,baseline={(0,0)}]
		\node[dot] (m1) at (-4,0) {};
		\node[dot] (m2) at (4,0) {};
		\draw[thick,out=270,in=0] (m1) to (-4.3,-1);
		\draw[thick,out=180,in=270] (-4.3,-1) to (-4.6,0);
		\draw[thick,out=90,in=180] (-4.6,0) to (-4.3,1) to [out=0,in=90] (m1) {};
		\draw[thick,out=270,in=180] (m2) to (4.3,-1);
		\draw[thick,out=0,in=270] (4.3,-1) to (4.6,0);
		\draw[thick,out=90,in=0] (4.6,0) to (4.3,1) to [out=180,in=90] (m2) {};
		\draw[thick] (-4.3,-1) to [out=15,in=180] (0,-2.5) to [out=0,in=165] (4.3,-1) {};
		\draw[thick,out=345,in=225] (-4.3,1) to (-3,2);
		\draw[thick,out=165,in=315] (4.3,1) to (3,2);
		\draw[dashed,out=330,in=180] (-3,2) to (0,1.5) to [out=0,in=210] (3,2); 
		\draw[dashed,out=30,in=180] (-3,2) to (0,2.5) to [out=0,in=150] (3,2);
		\draw[thick,bend right] (m1) to (m2) {};
		\node at (-1.0,-0.9) {$\alpha$};
		\node[dark purple] at (0.3,-1.9) {$\beta$};
		\draw[thick,dark purple,out=225,in=30] (m2) to (-1.2,-2.2) {};
		\draw[dashed,dark purple,out=120,in=300] (-1.2,-2.2) to (-3.75,0.95) {};
		\draw[thick,dark purple,out=270,in=60] (-3.85,1.06) to (m1) {};
		\draw[thick] (-0.75,0.3) to [out=345,in=195] (0.75,0.3) to [out=150,in=30] (-0.75,0.3) {}; 
		\draw[thick] (-0.75,0.3) to [out=165,in=345] (-0.9,0.35) {};
		\draw[thick] (0.75,0.3) to [out=15,in=195] (0.9,0.35) {};
	\end{tikzpicture}
	\caption{A covering pair $(\alpha,\beta)$ in a marked surface with only one marked point on each boundary component.}
	\label{fig: one marked point per boundary}
\end{figure}
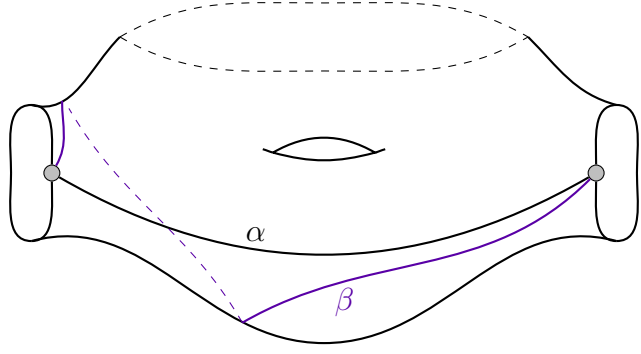

In the cases of $g\geq 1$ and $p\geq 1$ discussed above, the covering pair was chosen under the assumption that one boundary component contained at least two marked points. In this section, we consider the case where all boundary components contain a single marked point. 

Observe that the pair $(\alpha, \beta)$ in Figure \ref{fig: one marked point per boundary} is the same configuration as in Figure \ref{fig: covering pair}, so it is a covering pair by Lemma \ref{lemma: Muller LA 10.1}. Cutting $\Sigma_{g,b,m}^p$ along this pair will follow the same pattern as in Section \ref{section: unpunctured}, producing two copies of $\mathcal{A}(\Sigma_{g,b-1,m+2}^p/\sim ; \mathbb{F}_q)$, and the intersection of these two varieties will correspond to $q-1$ copies of $\mathcal{A}(\Sigma_{g,b-1,m+1}^p; \mathbb{F}_q)$. 

This gives us 
\begin{align*}
    &\#\mathcal{A}(\Sigma_{g,b,m}^p; \mathbb{F}_q) = 2\#\mathcal{A}(\Sigma_{g,b-1,m+2}^p/\sim; \mathbb{F}_q) - (q-1)\#\mathcal{A}(\Sigma_{g,b-1,m+1}^p; \mathbb{F}_q) \\
    &\quad= 2(q+1)^{2g+b-3} (q-1)^{2g+m+b-2} \big( q^{2g+m+b+p-1} (q^2+q-1)^p + (-1)^{m+2} (2q^2-1)^p \big) \\ &\quad\quad - (q+1)^{2g+b-3} (q-1)^{2g+m+b-1} \big( q^{2g+m+b+p-2} (q^2+q-1)^p + (-1)^{m+1} (2q^2-1)^p \big) \\
    &\quad= (q+1)^{2g+b-3} (q-1)^{2g+m+b-2} \big( 2q^{2g+m+b+p-1} (q^2+q-1)^p + 2(-1)^{m+2} (2q^2-1)^p  \\ &\quad\quad  - (q-1)(q^{2g+m+b+p-2} (q^2+q-1)^p + (-1)^{m+1} (2q^2-1)^p) \big) \\
    &\quad= (q+1)^{2g+b-3} (q-1)^{2g+m+b-2} \big( 2q^{2g+m+b+p-1} (q^2+q-1)^p + 2(-1)^{m+2} (2q^2-1)^p  \\ &\quad\quad  - q^{2g+m+b+p-1} (q^2+q-1)^p - q(-1)^{m+1} (2q^2-1)^p + q^{2g+m+b+p-2} (q^2+q-1)^p  \\ &\quad\quad + (-1)^{m+1} (2q^2-1)^p \big) \\
    &\quad= (q+1)^{2g+b-3} (q-1)^{2g+m+b-2} \big( q^{2g+m+b+p-1} (q^2+q-1)^p + (-1)^{m+2} (2q^2-1)^p  \\ &\quad\quad  + q(-1)^{m+2} (2q^2-1)^p + q^{2g+m+b+p-2} (q^2+q-1)^p \big) \\
    &\quad= (q+1)^{2g+b-2} (q-1)^{2g+m+b-2} \big( q^{2g+m+b+p-2} (q^2+q-1)^p + (-1)^{m} (2q^2-1)^p \big).    
\end{align*}
That is, we get the same formula as in Proposition \ref{prop: punctured count}, which implies that the order of cuts commutes. We can now incorporate this into the statement of the theorem.

\begin{thm}\label{thm: punctured count full}
    Let $\Sigma_{g,b,m}^p$ be a connected, triangulable marked surface of genus $g \geq 0$ with $b\geq 1$ boundary components, $m\geq 2$ marked points, and $p\geq 0$ punctures. 
	The corresponding cluster variety over $\mathbb{F}_q$ contains 
	\begin{equation}\label{eq: punctured count full}
        \#\mathcal{A}(\Sigma_{g,b,m}^p; \mathbb{F}_q) = (q+1)^{2g+b-2} (q-1)^{2g+m+b-2} \big( q^{2g+m+b+p-2} (q^2+q-1)^p + (-1)^m (2q^2-1)^p \big)
    \end{equation}
	points.
\end{thm}

Before moving on, we want to highlight a recurrence relation that appears when $q=2$, which is a generalization of a lemma of P{\'e}rez Melesio and Simental \cite[Lemma 1]{PMS26} to cluster algebras of surface type. 

\begin{lemma}\label{lemma: recurrence}
	The formula from Theorem \ref{thm: punctured count full} (Equation \ref{eq: punctured count full}) satisfies the following recurrence relation over $\mathbb{F}_2$:  
	\begin{equation}\label{eq: count recurrence} 
        \#\mathcal{A}(\Sigma_{g,b,m}^p; \mathbb{F}_2) = \#\mathcal{A}(\Sigma_{g,b,m-1}^p; \mathbb{F}_2) + 2 \#\mathcal{A}(\Sigma_{g,b,m-2}^p; \mathbb{F}_2).
    \end{equation}
\end{lemma}
\begin{proof}
	Over $\mathbb{F}_2$, we have  
	\[ \#\mathcal{A}(\Sigma_{g,b,m}^p; \mathbb{F}_2) = 3^{2g+b-2} \big( 2^{2g+m+b+p-2} 5^p + (-1)^m 7^p \big). \]
	Using this, we get 
	\begin{align*}
		&\#\mathcal{A}(\Sigma_{g,b,m-1}^p; \mathbb{F}_2) + 2 \#\mathcal{A}(\Sigma_{g,b,m-2}^p; \mathbb{F}_2) \\
        &\quad= 3^{2g+b-2} \big( 2^{2g+m+b+p-3} 5^p + (-1)^{m-1} 7^p \big) + 2\big(3^{2g+b-2} \big( 2^{2g+m+b+p-4} 5^p + (-1)^{m-2} 7^p \big)\big) \\
		&\quad= 3^{2g+b-2}\big( 2^{2g+m+b+p-3} 5^p + (-1)^{m-1} 7^p + 2\big(2^{2g+m+b+p-4} 5^p + (-1)^{m-2} 7^p\big) \big) \\
		&\quad= 3^{2g+b-2}\big( 2^{2g+m+b+p-2} 5^p + (-1)^m 7^p \big) \\
		&\quad= \#\mathcal{A}(\Sigma_{g,b,m}^p; \mathbb{F}_2). \qedhere
	\end{align*}	
\end{proof}

\begin{rem}
	By Lemma \ref{lemma: recurrence}, these point counts satisfy the recurrence relation  of the \emph{Jacobsthal sequence} (Sequence \seqnum{A001045} in the OEIS \cite{oeis}).
\end{rem}

\section{Counting algebraic tori}\label{section: tori}

Given a polygon $\Delta_{m+2}$ with $m+2$ sides, the number of triangulations is the $m$-th \emph{Catalan number} (Sequence \seqnum{A000108} in the OEIS \cite{oeis}),  \[ C_m = \frac{1}{m+1}\binom{2m}{m} = \frac{(2m)!}{m!(m+1)!}. \]
The number of triangulations is equal to the number of clusters in the corresponding cluster algebra. For each cluster $\mathbf{x}$, there is an inclusion of a corresponding algebraic torus $\mathbb{T}_\mathbf{x}$ into the cluster variety. The union of these \emph{cluster tori} over all triangulations of the surface is called the \emph{cluster manifold} $\mathcal{M}$. A point in the complement of the cluster manifold $\mathcal{A}(\Sigma_{g,b,m}^p ; \Bbbk) \setminus \mathcal{M}$ is called a \emph{deep point}; the set of all deep points is called the \emph{deep locus}.

For most of the surfaces we have considered, the number of triangulations is infinite, as are the numbers of cluster variables and clusters. However, because these cluster algebras are locally acyclic, the number of points in the corresponding cluster variety over $\mathbb{F}_q$ is finite. As a result, the cluster manifold can be covered by finitely many algebraic tori. 

\subsection{A naive upper bound using Catalan numbers}\label{sec: naive}

Let us denote the number of triangulations arising from the recursive cutting discussed in this note as $\tau (\Sigma)$. As already noted, $\tau (\Delta_{m}) = C_{m-2}$. 

Now, consider the annulus $\Sigma_{0,2,m}$. As discussed, the cluster variety is covered by two polygons, each with two additional marked points, so $ \tau (\Sigma_{0,2,m}) = 2C_m .$ 
With each additional boundary component, we will double the number of polygons necessary to cover the variety, and those polygons will have two additional marked points. Thus, we have \[ \tau (\Sigma_{0,b,m}) = 2^{b-1} C_{m + 2(b-1) - 2} = 2^{b-1} C_{m + 2b - 4} .\]

Next, we consider a surface with genus 1, $\Sigma_{1,b,m}$. The surface is covered by two genus 0 surfaces with an additional boundary component and two additional marked points, giving us \[ \tau (\Sigma_{1,b,m}) = 2 \tau (\Sigma_{0,b+1,m+2}) = 2^{b+1} C_{m + 2b} .\]
A surface with genus 2 would then be covered by two genus 1 surfaces with an additional boundary component and two additional marked points, so we have \[ \tau (\Sigma_{2,b,m}) = 2 \tau (\Sigma_{1,b+1,m+2}) = 2^{b+3} C_{m + 2b + 4} .\]
With each additional genus, four times as many polygons are required to cover the variety, each with four additional marked points, so we have 
\[ \tau (\Sigma_{g,b,m}) = 2^{2g+b-1} C_{m + 4g + 2b - 4}. \]

\begin{rem}
    Note that each of the polygons has $4g + 2b + 2p + m - 2$ sides, which is the expected number of sides for a cutting along a polygonal dissection \cite[Proposition 7.3]{BM25}. The number of polygons $2^{2g+b-1}$ is the number of polygonal dissections up to congruence, which are in bijection with several other related sets \cite[Theorem A.1]{BM25}. 
\end{rem}

Naively, we now attempt to extend this to punctured surfaces. First, consider a surface with a single puncture. From the cuttings of Section \ref{section: punctured surfaces}, we have 
\begin{align*}
\tau (\Sigma_{g,b,m}^{1}) &= \tau (\Sigma_{g,b,m+2}^{0}) + \big(\tau (\Sigma_{0,1,2}^{1})\big)\big(\tau (\Sigma_{g,b,m}^{0})\big) \\
&= \tau (\Sigma_{g,b,m+2}^{0}) + 4 \tau (\Sigma_{g,b,m}^{0})\\
&= 2^{2g+b-1} C_{m + 4g + 2b - 2} + 2^{2g+b+1} C_{m + 4g + 2b - 4} \\
&= 2^{2g+b-1} \big( C_{m + 4g + 2b - 2} + 4 C_{m + 4g + 2b - 4} \big).
\end{align*}

A surface with two punctures would likewise give us 
\begin{align*}
\tau (\Sigma_{g,b,m}^{2}) &= \tau (\Sigma_{g,b,m+2}^{1}) + \big(\tau (\Sigma_{0,1,2}^{1})\big)\big(\tau (\Sigma_{g,b,m}^{1})\big) \\
&= \tau (\Sigma_{g,b,m+2}^{1}) + 4\tau (\Sigma_{g,b,m}^{1}) \\
&= \tau (\Sigma_{g,b,m+4}^{0}) + 4\tau (\Sigma_{g,b,m+2}^{0}) + 4\tau (\Sigma_{g,b,m+2}^{0}) + 16\tau (\Sigma_{g,b,m}^{0}) \\
&= \tau (\Sigma_{g,b,m+4}^{0}) + 8\tau (\Sigma_{g,b,m+2}^{0})+ 16\tau (\Sigma_{g,b,m}^{0}).
\end{align*}

After working out a few more examples, it becomes clear that the coefficients follow the pattern of $(x+4)^p$, so we have 
\begin{align*}
    \tau (\Sigma_{g,b,m}^{p}) &= \sum_{j=0}^p \binom{p}{j} 4^{j} \tau (\Sigma_{g,b,m+2p-2j}^{0}) \\
    &= 2^{2g+b-1} \sum_{j=0}^p \binom{p}{j} 4^{j} C_{m + 4g + 2b +2p - 2j- 4}. 
\end{align*}

\begin{ex}\label{ex: D4 triangulations}
    Let $\Sigma_{0,1,4}^1$ be the once-punctured quadrilateral, which famously corresponds to the $D_4$ cluster algebra with four frozen boundary variables. The $D_4$ cluster algebra is known to have 50 seeds \cite[Proposition 5.9.1]{FWZChapters45}, but $\mathcal{A}(\Sigma_{0,1,4}^1; \mathbb{F}_2)$ contains just 29 points.
    \begin{align*}
        \tau (\Sigma_{0,1,4}^1) &= 2^{2g+b-1} \sum_{j=0}^p \binom{p}{j} 4^{j} C_{m + 4g + 2b +2p - 2j- 4} \\
        &= \sum_{j=0}^1 \binom{1}{j} 4^j C_{4-2j} \\
        &= C_4 + 4C_2 = 14 + 8 = 22.
    \end{align*}
    Observe that of the 29 points in $\mathcal{A}(\Sigma_{0,1,4}^1; \mathbb{F}_2)$, we know that seven are deep by previous work of the author \cite[Remark 5.22]{BeyerDeepIIarxiv} and Castronovo, Gorsky, Simental, and Speyer \cite[Remark 6.11]{CGSS26}. That is, just 22 points are contained in the cluster manifold.
\end{ex}

\begin{warn}
Unfortunately, Example \ref{ex: D4 triangulations} is quite misleading \textemdash the number of non-deep points over $\mathbb{F}_2$ is equal to $\tau (\Sigma)$ in very few cases. For a simple counterexample, $\tau(\Delta_6) = 14$, but $\mathcal{A}(\Delta_6; \mathbb{F}_2)$ contains just 11 points, 10 of which are not deep. In fact, as the complexity of $\Sigma$ increases, $\tau(\Sigma)$ grows much faster than the number of non-deep points over $\mathbb{F}_2$, so it quickly becomes a poor estimate of the number of algebraic tori necessary to cover the cluster manifold over $\mathbb{F}_2$.
\end{warn}

\subsection{Cuts along separating arcs}

Suppose we make a cut that disconnects a surface, as in Figure \ref{fig: disconnect}. In this figure, we see that a point that is not deep in $\mathcal{A}(\Sigma;\mathbb{F}_2)$ may be deep after cutting. Choosing the two arcs marked with $1$ as a covering pair, we see that $\mathcal{A}(\Delta_6;\mathbb{F}_2)$ is covered by $\mathcal{A}(\Delta_3 \sqcup \Delta_5; \mathbb{F}_2)$ and $\mathcal{A}(\Delta_4 \sqcup \Delta_4; \mathbb{F}_2)$. 
If we try to add triangulations from the cut polygon, we see that $\tau(\Delta_3)\tau(\Delta_5) + \tau(\Delta_4)\tau(\Delta_4) = 9$, but there are 10 non-deep points in $\mathcal{A}(\Delta_6;\mathbb{F}_2)$.

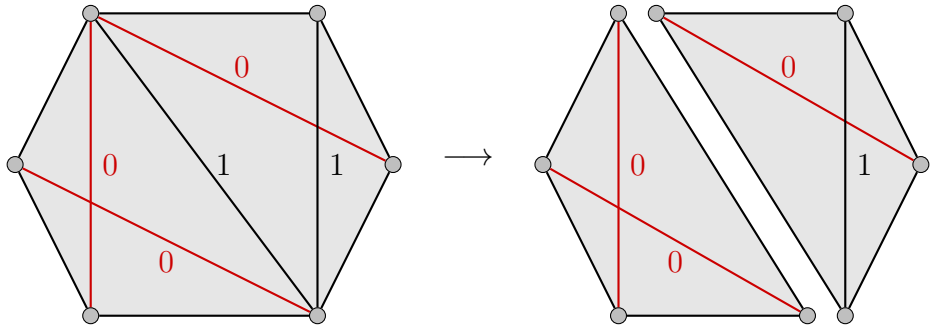
\begin{figure}[ht]
	\[
	\begin{tikzpicture}[scale=1.0,baseline=(current bounding box.center)]
		\path[fill=black!10] (-2,2) to (-1,2) to (0,0) to (-1,-2) to (-2,-2) to (-4,-2) to (-5,0) to (-4,2) to (-2,2);
		\node[dot] (1) at (-1,2) {};
		\node[dot] (2) at (0,0) {};
		\node[dot] (3) at (-1,-2) {};
		\node[dot] (4) at (-4,-2) {};
		\node[dot] (5) at (-5,0) {};
		\node[dot] (6) at (-4,2) {};
		\draw[thick] (1) to (2) (2) to (3) (3) to (4) (4) to (5) (5) to (6) (6) to (1);
		
		\draw[thick,dark red] (4) to node[dark red,right] {$0$} (6);
		\draw[thick,dark red] (5) to node[dark red,below] {$0$} (3);
		\draw[thick,dark red] (6) to node[dark red,above] {$0$} (2);
		\draw[thick] (6) to node[right] {$1$} (3);
		\draw[thick] (3) to node[right] {$1$} (1);
	\end{tikzpicture}
	\quad \longrightarrow \quad
	\begin{tikzpicture}[scale=1.0,baseline=(current bounding box.center)]
		\path[fill=black!10] (-3.5,2) to (-1,2) to (0,0) to (-1,-2) to (-3.5,2);
		\path[fill=black!10] (-5,0) to (-4,2) to (-1.5,-2) to (-4,-2) to (-5,0);
		\node[dot] (1) at (-1,2) {};
		\node[dot] (2) at (0,0) {};
		\node[dot] (3L) at (-1.5,-2) {};
		\node[dot] (3R) at (-1,-2) {};
		\node[dot] (4) at (-4,-2) {};
		\node[dot] (5) at (-5,0) {};
		\node[dot] (6L) at (-4,2) {};
		\node[dot] (6R) at (-3.5,2) {};
		\draw[thick] (1) to (2) (2) to (3R) (3R) to (6R) (6R) to (1);
		\draw[thick] (5) to (6L) (6L) to (3L) (3L) to (4) (4) to (5); 
		
		\draw[thick,dark red] (4) to node[dark red,right] {$0$} (6L);
		\draw[thick,dark red] (5) to node[dark red,below] {$0$} (3L);
		\draw[thick,dark red] (6R) to node[dark red,above] {$0$} (2);
		\draw[thick] (3R) to node[right] {$1$} (1);
	\end{tikzpicture}
	\]
	\caption{An example of a point over $\mathbb{F}_2$ that is not deep in the connected surface, but becomes deep after cutting.}
	\label{fig: disconnect}
\end{figure}

Taking the deep point of Figure \ref{fig: disconnect} into account, let us define a sequence as follows:
\begin{align}
	L_3 &\coloneq 1, \nonumber \\
	L_4 &\coloneq 2, \nonumber \\
	L_k &\coloneq L_{k-1} + 2L_{k-2} + 1. \label{eq: Lk relation}
\end{align}
We can view this sequence as a recursive cutting of $\Delta_m$ into $\Delta_3 \sqcup \Delta_{m-1}$ and $\Delta_{4} \sqcup \Delta_{m-2}$, with the additional 1 accounting for the deep point of $\mathcal{A}(\Delta_4 \sqcup \Delta_{m-2}; \mathbb{F}_2)$ which is not a deep point of $\mathcal{A}(\Delta_m; \mathbb{F}_2)$.

\begin{rem}\label{rem: Lichtenberg}
	The sequence $L_k$ defined here may be recognized as the Lichtenberg sequence (Sequence \seqnum{A000975} in the OEIS \cite{oeis}) with an offset of $2$. In particular, we can observe that $L_{2k} = 2L_{2k-1}$ and $L_{2k+1} = 2L_{2k}+1$.
\end{rem}

\begin{prop}\label{prop: polygon non-deep}
	Let $\Delta_m$ be a polygon with $m\geq 3$ sides. Then $L_m$ is exactly the number of non-deep points in $\mathcal{A}(\Delta_m; \mathbb{F}_2)$.
\end{prop}
\begin{proof}
	First, recall that the number of deep points of $\mathcal{A}(\Delta_m; \mathbb{F}_2)$ is exactly $1$ if $m$ is even, and $0$ if $m$ is odd \cite[Theorem 3.5]{BM25}. Simplifying the formulae for number of points over $\mathbb{F}_2$, the number of non-deep points in $\mathcal{A}(\Delta_m; \mathbb{F}_2)$ is
	\begin{displaymath}
	    \begin{cases}
    	 	\frac{1}{3} (2^{m-1}-1) &\text{if $m$ is odd,} \\
    		\frac{1}{3} (2^{m-1}+1)-1 &\text{if $m$ is even.}
    	\end{cases} 
    \end{displaymath}
	For $m=3$, we get $L_3 = 1$ non-deep points, and for $m=4$, we get $L_4 = 2$ non-deep points; for $m=5$, we get $L_5 = 5$ non-deep points. For $m\geq 6$, we will use induction while considering even and odd cases separately. 
	
	For our inductive hypothesis, assume there is some integer $k$ such that the proposition holds for all $3\leq i \leq k-1$. If $k$ is even, then the number of non-deep points is 
\begin{align*}
	\frac{1}{3} (2^{k-1}+1) - 1 &= \frac{1}{3}(2^{k-1} - 2) = \frac{2}{3}(2^{k-2} - 1) \\
    &= 2L_{k-1} = L_{k-1} + 2 L_{k-2} + 1 = L_k.
\end{align*}
	If $k$ is odd, then the number of non-deep points is 
\begin{align*}
	\frac{1}{3} (2^{k-1}-1) &= \frac{1}{3}(2^{k-1} + 2) - 1 = \frac{2}{3}(2^{k-2} + 1) - 1 = 2\bigg(\frac{1}{3}(2^{k-2} + 1) - 1\bigg) + 1 \\
    &= 2L_{k-1} + 1 = L_{k-1} + 2 L_{k-2} + 1 = L_k.
\end{align*}
Thus, $L_k$ is the number of non-deep points in $\mathcal{A}(\Delta_m; \mathbb{F}_2)$ for all $m\geq 3$. 
\end{proof}

\subsection{Algebraic tori over \texorpdfstring{$\mathbb{F}_2$}{F2}}

 From Theorem \ref{thm: punctured count full}, we know the total number of points in a cluster variety over $\mathbb{F}_2$ is 
\[ \#\mathcal{A}(\Sigma_{g,b,m}^p; \mathbb{F}_2) = 3^{2g+b-2} \big( 2^{2g+m+b+p-2} 5^p + (-1)^m 7^p \big). \]
By subtracting the number of deep points, we can find the number of points contained in the cluster manifold over $\mathbb{F}_2$. We can then adjust the recurrence relation of $L_k$ (Equation \ref{eq: Lk relation}) to extend Proposition \ref{prop: polygon non-deep} to more general marked surfaces.

\subsubsection{Unpunctured Surfaces}

For an unpunctured surface $\Sigma_{g,b,m}^0$, there are no deep points if the number of marked points $m$ is odd, and there are $2^{2g+b-1}$ deep points over $\mathbb{F}_2$ if $m$ is even \cite[Theorems 7.12 and A.1]{BM25}. As such, we know the number of non-deep points is equal to 
\begin{equation}\label{eq: L unpunctured}
	L(\Sigma_{g,b,m}^0) = 
	\begin{cases}
		\#\mathcal{A}(\Sigma_{g,b,m}^0; \mathbb{F}_2) & \text{if $m$ is odd,} \\
		\#\mathcal{A}(\Sigma_{g,b,m}^0; \mathbb{F}_2) - 2^{2g+b-1} & \text{if $m$ is even.}
	\end{cases}
\end{equation}

\begin{prop}\label{prop: unpunc recursion}
    Let $\Sigma_{g,b,m}^0$, $\Sigma_{g,b,m-1}^0$, and $\Sigma_{g,b,m-2}^0$ be connected, unpunctured, triangulable marked surfaces of genus $g\geq 0$ with $b\geq 1$ boundary components and $m-2\geq 2$ marked points. 
	Then the number of non-deep points over $\mathbb{F}_2$ satisfies the recurrence relation 
	\begin{equation}\label{eq: unpunc non-deep recur}
       L(\Sigma_{g,b,m}^0) = L(\Sigma_{g,b,m-1}^0) + 2 L(\Sigma_{g,b,m-2}^0) + 2^{2g+b-1}. 
    \end{equation}
\end{prop}
\begin{proof}
	If $m$ is odd, then using Lemma \ref{lemma: recurrence}, we have 
\begin{align*}
	&L(\Sigma_{g,b,m-1}^0) + 2 L(\Sigma_{g,b,m-2}^0) + 2^{2g+b-1} \\
    &\quad= \#\mathcal{A}(\Sigma_{g,b,m-1}^0; \mathbb{F}_2) - 2^{2g+b-1} + 2\#\mathcal{A}(\Sigma_{g,b,m-2}^0; \mathbb{F}_2) + 2^{2g+b-1} \\
	&\quad= \#\mathcal{A}(\Sigma_{g,b,m-1}^0; \mathbb{F}_2) + 2 \#\mathcal{A}(\Sigma_{g,b,m-2}^0; \mathbb{F}_2) \\
	&\quad= \#\mathcal{A}(\Sigma_{g,b,m}^0; \mathbb{F}_2) \\
    &\quad= L(\Sigma_{g,b,m}^0).
\end{align*}

	Similarly, if $m$ is even, we have 
\begin{align*}
	&L(\Sigma_{g,b,m-1}^0) + 2 L(\Sigma_{g,b,m-2}^0) + 2^{2g+b-1} \\
    &\quad= \#\mathcal{A}(\Sigma_{g,b,m-1}^0; \mathbb{F}_2) + 2\big(\#\mathcal{A}(\Sigma_{g,b,m-2}^0; \mathbb{F}_2) - 2^{2g+b-1}\big) + 2^{2g+b-1} \\
	&\quad= \#\mathcal{A}(\Sigma_{g,b,m-1}^0; \mathbb{F}_2) + 2\#\mathcal{A}(\Sigma_{g,b,m-2}^0; \mathbb{F}_2) - 2^{2g+b-1} \\
	&\quad= \#\mathcal{A}(\Sigma_{g,b,m}^0; \mathbb{F}_2) - 2^{2g+b-1} \\
    &\quad= L(\Sigma_{g,b,m}^0).\qedhere
\end{align*}
\end{proof}

\subsubsection{Punctured Surfaces}

For punctured surfaces, calculating the number of deep points is more complicated. For a given puncture, a deep point that kills all arcs (of both taggings) emanating from that puncture is said to \emph{quash} that puncture. A deep point that quashes all punctures is called a \emph{malevolent} point, and a deep point that does not quash any punctures is called \emph{benevolent}. 

Benevolent points exist only when the number of marked points $m$ is even. To quantify the benevolent points over $\mathbb{F}_2$, we have the following theorem:
\begin{thm}\cite[Theorem 7.9]{BeyerDeepIIarxiv}
	Let $\SS$ be a connected, triangulable marked surface with genus $g\geq 0$, $b\geq 1$ boundary components, $m\geq 2$ marked points on the boundary, and $q\geq 1$ punctures. 
	\begin{enumerate}
		\item If $m$ is odd, then $V(\CA(\SS),\Bbbk)$ has no benevolent deep points. 
		\item If $m$ is even, then the deep locus of $V(\CA(\SS),\Bbbk)$ contains $2^{2g+b+2q-1}$ disjoint sets of benevolent points, each isomorphic to $(\Bbbk^\times)^{2g+b+q+m-2}$.
	\end{enumerate}
\end{thm}

\begin{cor}\label{cor: benevolent F2}
	Let $\Sigma_{g,b,m}^p$ be a connected, triangulable marked surface with genus $g\geq 0$, $b\geq 1$ boundary components, $m\geq 2$ marked points on the boundary, and $q\geq 1$ punctures. 
	If $m$ is odd, then $\mathcal{A}(\Sigma_{g,b,m}^p ; \mathbb{F}_2)$ has no benevolent deep points. If $m$ is even, then the deep locus of $\mathcal{A}(\Sigma_{g,b,m}^p ; \mathbb{F}_2)$ contains $2^{2g+b+2p-1}$  benevolent points.
\end{cor}

On the other hand, deep points that quash some number of punctures exist regardless of whether $m$ is even or odd. We know that the set of deep points that quash a given puncture $\xi_i$ is isomorphic to $\mathcal{A}(\Sigma_{g,b,m}^{p-1}; \Bbbk)$; the intersection of any two such sets is isomorphic to $\mathcal{A}(\Sigma_{g,b,m}^{p-2}; \Bbbk)$, and the intersection of any $k$ such sets is isomorphic to $\mathcal{A}(\Sigma_{g,b,m}^{p-k}; \Bbbk)$ (as long as $\Sigma_{g,b,m}^p$ is not a punctured bigon) \cite[Section 7.1]{BeyerDeepIIarxiv}. Combined with the number of benevolent points from Corollary \ref{cor: benevolent F2}, the inclusion-exclusion principle gives us the following formula for the number of deep points over $\mathbb{F}_2$.

\begin{lemma}\label{lemma: deep count formula}
	Let $\Sigma_{g,b,m}^p$ be a connected, triangulable marked surface of genus $g$ with $b\geq 1$ boundary components, $m\geq 2$ marked points, and $p\geq 1$ punctures, such that $\Sigma_{g,b,m}^p$ is not a punctured bigon. 
	The number of deep points in $\mathcal{A}(\Sigma_{g,b,m}^p ; \mathbb{F}_2)$ is
\begin{equation}\label{eq: deep count}
	D(\Sigma_{g,b,m}^p) \coloneq 
	\begin{cases}
		\displaystyle \sum_{i=1}^p (-1)^{i+1}\binom{p}{i} \#\mathcal{A}(\Sigma_{g,b,m}^{p-i} ; \mathbb{F}_2) & \text{if $m$ is odd,} \\
		2^{2g+b+2p-1} + \displaystyle \sum_{i=1}^p (-1)^{i+1}\binom{p}{i} \#\mathcal{A}(\Sigma_{g,b,m}^{p-i} ; \mathbb{F}_2) & \text{if $m$ is even.}
	\end{cases}
\end{equation}	
\end{lemma}

Note that this formula for the number of deep points over $\mathbb{F}_2$ satisfies the following recurrence relation.
\begin{lemma}\label{lem: punc deep recursion}
	Equation \ref{eq: deep count} satisfies the following relation: 
	\begin{equation}\label{eq: punc deep recur} 
        D(\Sigma_{g,b,m}^p) + 2^{2g+b+2p-1} = D(\Sigma_{g,b,m-1}^p) + 2 D(\Sigma_{g,b,m-2}^p).
    \end{equation}	
\end{lemma}
\begin{proof}
	If $m$ is odd, then using Lemma \ref{lemma: recurrence}, 
	\begin{align*}
		&D(\Sigma_{g,b,m-1}^p) + 2 D(\Sigma_{g,b,m-2}^p) \\ 
        &\quad= 2^{2g+b+2p-1} + \sum_{i=1}^p (-1)^{i+1}\binom{p}{i} \#\mathcal{A}(\Sigma_{g,b,m-1}^{p-i} ; \mathbb{F}_2) + 2 \sum_{i=1}^p (-1)^{i+1}\binom{p}{i} \#\mathcal{A}(\Sigma_{g,b,m-2}^{p-i} ; \mathbb{F}_2) \\
		&\quad= 2^{2g+b+2p-1} + \sum_{i=1}^p (-1)^{i+1}\binom{p}{i} \big( \#\mathcal{A}(\Sigma_{g,b,m-1}^{p-i} ; \mathbb{F}_2) + 2 \#\mathcal{A}(\Sigma_{g,b,m-2}^{p-i} ; \mathbb{F}_2) \big) \\
		&\quad= 2^{2g+b+2p-1} + \sum_{i=1}^p (-1)^{i+1}\binom{p}{i} \#\mathcal{A}(\Sigma_{g,b,m}^{p-i} ; \mathbb{F}_2) \\
		&\quad=  2^{2g+b+2p-1} + D(\Sigma_{g,b,m}^p).
	\end{align*}	

Similarly, if $m$ is even,
\begin{align*}
	&D(\Sigma_{g,b,m-1}^p) + 2 D(\Sigma_{g,b,m-2}^p) \\
    &\quad = \sum_{i=1}^p (-1)^{i+1}\binom{p}{i} \#\mathcal{A}(\Sigma_{g,b,m-1}^{p-i} ; \mathbb{F}_2) + 2 \bigg( 2^{2g+b+2p-1} + \sum_{i=1}^p (-1)^{i+1}\binom{p}{i} \#\mathcal{A}(\Sigma_{g,b,m-2}^{p-i} ; \mathbb{F}_2) \bigg) \\
	&\quad= 2^{2g+b+2p} + \sum_{i=1}^p (-1)^{i+1}\binom{p}{i} \big( \#\mathcal{A}(\Sigma_{g,b,m-1}^{p-i} ; \mathbb{F}_2) + 2 \#\mathcal{A}(\Sigma_{g,b,m-2}^{p-i} ; \mathbb{F}_2) \big) \\
	&\quad= 2^{2g+b+2p} + \sum_{i=1}^p (-1)^{i+1}\binom{p}{i} \#\mathcal{A}(\Sigma_{g,b,m}^{p-i} ; \mathbb{F}_2) \\
	&\quad= 2^{2g+b+2p-1} + D(\Sigma_{g,b,m}^p).\qedhere
\end{align*}	
\end{proof}

We can now construct a formula for the number of non-deep points over $\mathbb{F}_2$, 
\begin{equation}\label{eq: non-deep count}
 L(\Sigma_{g,b,m}^p) \coloneqq \#\mathcal{A}(\Sigma_{g,b,m}^p; \mathbb{F}_2) - D(\Sigma_{g,b,m}^p), 
\end{equation}
and a corresponding recurrence relation.
\begin{thm}\label{thm: nondeep recurrence}
    Let $\Sigma_{g,b,m}^p$, $\Sigma_{g,b,m-1}^p$, and $\Sigma_{g,b,m-2}^p$ be connected, triangulable marked surfaces of genus $g$ with $b\geq 1$ boundary components, $m-2 \geq 2$ marked points, and $p\geq 1$ punctures, such that none are punctured bigons. 
    Then the number of non-deep points over $\mathbb{F}_2$ (Equation \ref{eq: non-deep count}) satisfies the recurrence relation 
    \begin{equation}\label{eq: non-deep recurrence}
	   L(\Sigma_{g,b,m}^p) = L(\Sigma_{g,b,m-1}^p) + 2 L(\Sigma_{g,b,m-2}^p) + 2^{2g+b+2p-1}. 
    \end{equation}
\end{thm}
\begin{proof}
	If $p=0$, this follows from Proposition \ref{prop: unpunc recursion}. If $p \geq 1$, using Lemma \ref{lemma: recurrence} and Lemma \ref{lem: punc deep recursion}, we have 
	\begin{align*}
		&L(\Sigma_{g,b,m-1}^p) + 2 L(\Sigma_{g,b,m-2}^p) \\
        &\quad= \#\mathcal{A}(\Sigma_{g,b,m-1}^p; \mathbb{F}_2) - D(\Sigma_{g,b,m-1}^p) + 2\#\mathcal{A}(\Sigma_{g,b,m-2}^p; \mathbb{F}_2) - 2D(\Sigma_{g,b,m-2}^p)\\
		&\quad= \#\mathcal{A}(\Sigma_{g,b,m}^p; \mathbb{F}_2) - \big( D(\Sigma_{g,b,m-1}^p) + 2D(\Sigma_{g,b,m-2}^p) \big) \\
		&\quad= \#\mathcal{A}(\Sigma_{g,b,m}^p; \mathbb{F}_2) - 2^{2g+b+2p-1} - D(\Sigma_{g,b,m}^p) \\ 
		&\quad= L(\Sigma_{g,b,m}^p) - 2^{2g+b+2p-1}.\qedhere
	\end{align*}
\end{proof}

\section*{Acknowledgments}

The author would like to thank Jos{\'e} Simental for helpful conversations and comments on earlier versions of this note. The author would also like to thank Thomas Kahle for pointing out the \verb|enum_affine_finite_field| function in SageMath, which was used to verify several of the smaller cases addressed in this note. 

\subsection*{Funding}

This work was supported by Universidad Nacional Aut{\'o}noma de M{\'e}xico Postdoctoral Program. 

\subsection*{AI}

The author has not used artificial intelligence at any stage of the preparation of this note.

\bibliographystyle{plain}
\bibliography{main}

\end{document}